\documentclass[a4paper, 12pt]{{amsart}}
\usepackage{amsmath, mathtools}
\usepackage{amscd}
\usepackage{amssymb}
\usepackage{amsthm}
\usepackage{ascmac, color}
\usepackage[dvipdfmx]{graphicx, xcolor, pdfpages}
\usepackage[all]{xy}

\newtheorem{lemma}{{\sc Lemma}}[section]
\newtheorem{corollary}[lemma]{{\sc Corollary}}
\newtheorem{proposition}[lemma]{{\sc Proposition}}
\newtheorem{theorem}[lemma]{{\sc Theorem}}
\theoremstyle{definition}
\newtheorem{remark}[lemma]{{\sc Remark}}

\numberwithin{equation}{section}

\def\Gb{{\mathfrak{b}}}
\def\Gc{{\mathfrak{c}}}
\def\Gg{{\mathfrak{g}}}
\def\Gh{{\mathfrak{h}}}

\def\Gn{{\mathfrak{n}}}

\def\GM{{\mathfrak{M}}}
\def\GN{{\mathfrak{N}}}

\def\BC{{\mathbf{C}}}

\def\BF{{\mathbf{F}}}

\def\BQ{{\mathbf{Q}}}

\def\BZ{{\mathbf{Z}}}

\def\CB{{\mathcal B}}

\def\DD{{\mathcal D}}

\def\CO{{\mathcal O}}

\def\CL{{\mathcal L}}

\def\ad{{\mathop{\rm ad}\nolimits}}

\def\Comod{\mathop{\rm Comod}\nolimits}
\def\deru{\partial}

\def\End{\mathop{\rm{End}}\nolimits}

\def\eq{\mathop{\rm eq}\nolimits}

\def\For{{\mathop{\rm For}\nolimits}}
\def\gr{{\mathrm gr}}
\def\Hom{\mathop{\rm Hom}\nolimits}

\def\id{\mathop{\rm id}\nolimits}
\def\Id{\mathop{\rm Id}\nolimits}
\def\Ind{{\mathop{\rm Ind}\nolimits}}
\def\inte{{\mathop{\rm int}\nolimits}}
\def\Image{\mathop{\rm Im}\nolimits}
\def\Ker{\mathop{\rm Ker\hskip.5pt}\nolimits}

\def\Mod{\mathop{\rm Mod}\nolimits}

\def\op{{\mathop{\rm op}\nolimits}}

\def\Proj{\mathop{\rm Proj}\nolimits}

\def\res{{\mathop{\rm res}\nolimits}}

\def\simto{\xrightarrow{\sim}}

\def\Tor{{\rm{Tor}}}

\def\eq{{\rm{eq}}}
\def\tU{\widetilde{U}}
\def\tGamma{{\Breve{\Gamma}}}
\def\Df{D^{f}}

\def\tDD{{\widetilde\DD}}
\def\tDDf{{\widetilde\DD}^f}
\def\zDf{{}^0\!D^f}
\def\iDf{{}^1\!D^f}
\def\tDDzf{{}^0{\widetilde\DD}^{f}}

\def\EN{{\overline{E}_R^{N_q}}}
\def\Uac{\triangleleft}
\def\rac{\blacktriangleleft}
\def\lac{\blacktriangleright}
\def\Loc{\CL oc}
\def\tchi{\tilde{\chi}}
\def\Uf{U_q^f}
\def\Ue{U_q^e}

\begin{document}
\title[$D$-modules on a quantized flag manifold]
{
Categories of $D$-modules \\
on a quantized flag manifold}
\author{Toshiyuki TANISAKI}
\begin{abstract}
There are two approaches in defining the category of $D$-modules on a quantized flag manifold.
One is due to Lunts and Rosenberg based on the $\Proj$-construction of the quantized flag manifold, 
and the other is due to Backelin and Kremnizer 
using equivariant $D$-modules on the corresponding quantized algebraic group.
In this paper we compare the two approaches.

\end{abstract}
\maketitle
\section{Introduction}
\subsection{}
Let $G$ be a connected, simply-connected simple algebraic group over the complex number field $\BC$, and let $B$ be a Borel subgroup of $G$.
Representations of the Lie algebra $\Gg$ of $G$ are realized using $D$-modules on the flag manifold $\CB=B\backslash G$ 
through the works of Beilinson and Bernstein  \cite{BB} and Brylinski and Kashiwara \cite{BrK}.
In this paper we are concerned with the analogue of this theory for quantum groups.
\subsection{}
The quantized flag manifold $\CB_q$ is an non-commutative algebraic variety.
By a general philosophy of the non-commutative geometry, 
giving a non-commutative algebraic variety $X$ is 
the same as giving
a category 
(denoted by $\Mod(\CO_{X})$) which can be regarded as  the category of quasi-coherent sheaves on $X$.

One way to define $\Mod(\CO_{\CB_q})$ is to use 
the $\Proj$-construction.
Note that the ordinary flag manifold $\CB$ is a projective algebraic variety
written as $\CB=\Proj(A)$
for a graded (commutative) algebra 
$A=O(N\backslash G)$,
where $N$ is the unipotent radical of $B$ and 
$O(N\backslash G)$ is the  subalgebra of the affine coordinate algebra $O(G)$ of $G$ consisting of 
left $N$-invariant functions.
Using the quantized enveloping algebra $U_q(\Gg)$ 
we can construct  natural quantum analogues  $A_q=O(N_q\backslash G_q)\subset O(G_q)$ of $A=O(N\backslash G)\subset O(G)$.
Then the theory of non-commutative projective schemes,
which was  developed by Artin and Zhang \cite{AZ}, Ver\"{e}vkin \cite{V}, Rosenberg \cite{R} following Manin's idea \cite{M},
allows us to define a category 
$\Mod(\CO_{\CB_q})$ using 
the non-commutative graded ring 
$A_q$, by which we can write
$\CB_q=\Proj(A_q)$
(see Lunts and Rosenberg \cite{LR}).

Another way to construct $\Mod(\CO_{\CB_q})$ was given by Backelin and Kremnizer.
Let $p:G\to\CB$ be the natural morphism.
The pull-back functor 
$p^*:\Mod(\CO_\CB)\to\Mod(\CO_G)$ induces the equivalence
$
\Mod(\CO_\CB)\cong\Mod(\CO_G,B)
$
of categories, where $\Mod(\CO_G,B)$ denotes the category of $B$-equivariant quasi-coherent $\CO_G$-modules.
Since $G$ is an affine algebraic variety, $\Mod(\CO_G,B)$ is naturally identified with the category 
 of $O(G)$-modules equipped with an $O(B)$-comodule structure satisfying a certain compatibility condition.
Using the quantum analogues $O(G_q)$, 
$O(B_q)$ of $O(G)$, $O(B)$
we obtain a natural quantum analogue 
$\Mod(\CO_{G_q},B_q)$ of the category 
$\Mod(\CO_{G},B)$.
In fact Backelin and Kremnizer have shown the equivalence 
\begin{equation}
\Mod(\CO_{\CB_q})\cong\Mod(\CO_{G_q},B_q)
\end{equation}
of categories
using a characterization of $\Mod(\CO_{\CB_q})$ by Aritin and Zhang \cite{AZ}.
\subsection{}
The main theme of this paper is the theory of $D$-modules on the quantized flag manifold.
We will compare the two approaches, 
one by Lunts and Rosenberg \cite{LR} 
and the other by Backelin and Kremnizer 
\cite{BK}.
We point out that there is  also another  earlier approach by Joseph \cite{Jo0}.

Lunts and Rosenberg \cite{LRD} developed a theory of $D$-modules on a general non-commutative projective scheme $X=\Proj(R)$.
Using a certain ring of differential operators $D^\dagger(R)\subset\End(R)$, 
they defined a category $\Mod(\DD^\dagger_{X,\gamma})$, which can be regarded as an analogue of the category of modules over the ring of twisted differential operators, 
where $\gamma$ is a datum for the twist.
In \cite{LR} they applied it to $\CB_q$, 
and considered the category 
$\Mod(\DD^\dagger_{\CB_q,\lambda})$
for integral weights $\lambda$.
This category is related to $U_q(\Gg)$-modules via the global section functor 
\begin{equation}
\label{eq:IGamma1}
\Gamma:\Mod(\DD^\dagger_{\CB_q,\lambda})
\to
\Mod(U_q(\Gg),\lambda),
\end{equation}
where $\Mod(U_q(\Gg),\lambda)$ denotes the category of $U_q(\Gg)$-modules with central character associated to 
$\lambda$.
They conjectured that \eqref{eq:IGamma1} gives an equivalence when $\lambda$ is dominant.

In \cite{T0} we 
considered a modified category 
$\Mod(\DD_{\CB_q,\lambda})$
using
a subring
$D(A_q)$ of $D^\dagger(A_q)$ 
generated by the left multiplication of elements of $A_q$, the natural action of $U_q(\Gg)$, and the grading operators.
We have also the global section functor
\begin{equation}
\label{eq:IGamma2}
\Gamma:\Mod(\DD_{\CB_q,\lambda})
\to
\Mod(U_q(\Gg),\lambda).
\end{equation}
We proved that 
that \eqref{eq:IGamma2} gives an equivalence when $\lambda$ is dominant.

On the other hand Backelin and Kremnizer 
\cite{BK} developed the theory of $D$-modules on the quantized flag manifold
based on  $\Mod(\CO_{G_q}, B_q)$.
We have a natural $U_q(\Gg)$-bimodule structure of $O(G_q)$, where the left 
(resp.\ right) action of $U_q(\Gg)$ is the quantum analogue of the left 
(resp.\ right) action of $U(\Gg)$ on $O(G)$ induced by the right 
(resp.\ left) multiplication of $G$ on $G$.
Let $D(G_q)$ be  the ring of differential operators acting on $O(G_q)$
generated by the left multiplication of elements of $O(G_q)$ and the right action of $U_q(\Gg)$.
For an integral weight $\lambda$ Backelin and Kremnizer considered the category $\Mod(\DD_{G_q},B_q,\lambda)$ 
consisting of $D(G_q)$-modules equipped with an $O(B_q)$-comodule structure satisfying certain conditions.
Let $\Uf(\Gg)$ be the subalgebra of $U_q(\Gg)$ consisting of $\ad$-finite elements, and let 
$\Mod(\Uf(\Gg),\lambda)$ be the category of $\Uf(\Gg)$-modules with central character associated to $\lambda$.
We have a functor
\begin{equation}
\label{eq:BKBB}
(\bullet)^{B_q}:
\Mod(\DD_{G_q},B_q,\lambda)
\to\Mod(\Uf(\Gg),\lambda)
\end{equation}
sending $M\in \Mod(\DD_{G_q},B_q,\lambda)$ to the subspace $M^{B_q}$ 
of $M$ consisting of $O(B_q)$-coinvariant elements.
Here the action of $\Uf(\Gg)$ on $M^{B_q}$  is induced by the ring homomorphism $\Uf(\Gg)\to D(G_q)$
given by the left (not right!) action of $\Uf(\Gg)$ on $O(G_q)$.

In this paper we 
consider the category 
$\Mod(\DD^f_{\CB_q,\lambda})$ 
defined similarly to 
$\Mod(\DD_{\CB_q,\lambda})$, 
using a subring 
$\Df(A_q)$ of $D(A_q)$ 
generated by left multiplication of elements of $A_q$, the natural action of $\Uf(\Gg)$, and the grading operators.
We have the global section functor
\begin{equation}
\label{eq:IGamma3}
\Gamma:\Mod(\DD^f_{\CB_q,\lambda})
\to
\Mod(\Uf(\Gg),\lambda).
\end{equation}
Our main theorem is the equivalence
\begin{equation}
\label{eq:I-equiv}
\Mod(\DD^f_{\CB_q,\lambda})
\cong
\Mod(\DD_{G_q},B_q,\lambda)
\end{equation}
of categories for any integral weights $\lambda$.
Similarly to the case of 
$\Mod(\DD_{\CB_q,\lambda})$ 
we can show
that \eqref{eq:IGamma3} gives an equivalence when $\lambda$ is dominant.
Since the equivalence \eqref{eq:I-equiv} 
is compatible with \eqref{eq:BKBB} and \eqref{eq:IGamma3}, 
this recovers  the equivalence of \eqref{eq:BKBB} 
for dominant integral weights $\lambda$
in \cite{BK}.

\subsection{}
\label{subsec:notation}
For a ring $R$ we denote by $R^\op$ the ring opposite to it.
Namely, we have $R^\op=\{r^\circ\mid r\in R\}$ with
$r^\circ_1+r^\circ_2=(r_1+r_2)^\circ$, 
$r^\circ_1r^\circ_2=(r_2r_1)^\circ$ for
$r_1, r_2\in R$.

For a ring $R$ we denote by $\Mod(R)$ the category of left $R$-modules.

For a graded ring $R$ graded by an abelian group $\Gamma$ we denote by 
$\Mod_\Gamma(R)$ the category of $\Gamma$-graded left $R$-modules.
Namely, an object of $\Mod_\Gamma(R)$ 
is a left $R$-module $M$ equipped with the decomposition $M=\bigoplus_{\gamma\in\Gamma}M(\gamma)$ satisfying 
$R(\gamma)M(\delta)\subset M(\gamma+\delta)$ for $\gamma, \delta\in\Gamma$, 
and for $M, N\in\Mod_\Gamma(R)$ a morphism from $M$ to $N$ is a homomorphism $f:M\to N$ of $R$-modules satisfying $f(M(\gamma))\subset N(\gamma)$ for any $\gamma\in\Gamma$.
For $M, N\in\Mod_\Gamma(R)$ 
we denote  by $\Hom_{R}^\gr(M,N)$ the set of morphisms from $M$ to $N$.

For $\gamma\in\Gamma$ we define a functor 
\[
(\bullet)[\gamma]:\Mod_\Gamma(R)\to \Mod_\Gamma(R)
\qquad(M\mapsto M[\gamma])
\]
by 
$(M[\gamma])(\delta)=M(\gamma+\delta)$ for $\delta\in \Gamma$.

For a Hopf algebra $H$ we use Sweedler's notation 
\[
\Delta(h)=\sum_{(h)}h_{(0)}\otimes h_{(1)}
\]
for the 
comultiplication $\Delta:H\to H\otimes H$.
\section{Quantum groups}
\subsection{}
Let $G$ be a connected, simply-connected simple algebraic group over the complex number field $\BC$, 
and let $H$ be a maximal torus of $G$.
We denote by $\Gg$, $\Gh$ the Lie algebras of $G$ and $H$ respectively.
Let $\Delta\subset \Gh^*$ be the set of roots, and 
let $W\subset GL(\Gh^*)$ be the Weyl group.
We denote by 
$(\,,\,)$ the $W$-invariant symmetric bilinear form on $\Gh^*$ such that $(\alpha,\alpha)=2$ for any short root $\alpha\in\Delta$.
For $\alpha\in\Delta$ we set $\alpha^\vee=2\alpha/(\alpha,\alpha)$.
Let $\Lambda\subset\Gh^*$ and $Q\subset \Gh^*$ be the weight lattice and the root lattice respectively.
We fix a set of simple roots $\{\alpha_i\mid i\in I\}$ and denote by $\Delta^+$ the  set of positive roots.
We set 
\[
Q^+=\sum_{\alpha\in\Delta^+}\BZ_{\geqq0}\alpha,
\qquad
\Lambda^+=
\{\lambda\in\Lambda\mid(\lambda,\alpha^\vee)\geqq0\;(\forall\alpha\in\Delta^+)\}.
\]
We define subalgebras $\Gn$, $\Gn^+$, $\Gb$, $\Gb^+$ of $\Gg$ by
\[
\Gn=\bigoplus_{\alpha\in\Delta^+}\Gg_{-\alpha},
\qquad
\Gn^+=\bigoplus_{\alpha\in\Delta^+}\Gg_{\alpha},
\qquad
\Gb=\Gh\oplus\Gn,
\qquad
\Gb^+=\Gh\oplus\Gn^+,
\]
where $\Gg_\alpha$ for $\alpha\in\Delta$ denotes the corresponding root subspace.
We denote by $N$, $B$ the closed subgroups of $G$  corresponding to $\Gn$, $\Gb$ respectively.
\subsection{}
For $i\in I$ we set $d_i=(\alpha_i,\alpha_i)/2$, 
and for $i, j\in I$ we set
$a_{ij}=(\alpha_j,\alpha_i^\vee)$.
We fix a positive integer $N$ such that $N(\Lambda,\Lambda)\subset\BZ$.
The quantized enveloping algebra $U_q(\Gg)$ is the associative algebra over the rational function field
\begin{equation}
\BF=\BQ(q^{1/N})
\end{equation}
generated by the elements
\[
k_\lambda\;\;(\lambda\in\Lambda),
\qquad
e_i, \; f_i\;\;(i\in I)
\]
satisfying the relations
\begin{align}
&k_0=1,\qquad
k_\lambda k_\mu=k_{\lambda+\mu}
&(\lambda,\mu\in\Lambda),
\\
&k_\lambda e_i=
q_i^{(\lambda,\alpha^\vee_i)}e_ik_\lambda,
\qquad
k_\lambda f_i=
q_i^{-(\lambda,\alpha^\vee_i)}f_ik_\lambda
&(\lambda\in\Lambda, i\in I),
\\
&
e_if_j-f_je_i=\delta_{ij}\frac{k_i-k_i^{-1}}{q_i-q_i^{-1}}
&(i, j\in I),
\\
&
\sum_{r=0}^{1-a_{ij}}
(-1)^r
\begin{bmatrix}
1-a_{ij}
\\
r
\end{bmatrix}_{q_i}
e_i^{1-a_{ij}-r}e_je_i^r=0
&(i, j\in I, \; i\ne j),
\\
&
\sum_{r=0}^{1-a_{ij}}
(-1)^r
\begin{bmatrix}
1-a_{ij}
\\
r
\end{bmatrix}_{q_i}
f_i^{1-a_{ij}-r}f_jf_i^r=0
&(i, j\in I, \; i\ne j).
\end{align}
Here, $q_i=q^{d_i}$, $k_i=k_{\alpha_i}$ for $i\in I$, 
and 
\[
\begin{bmatrix}
m
\\
n
\end{bmatrix}_{t}
=
\prod_{a=1}^{n}
\frac{t^{m+1-a}-t^{-(m+1-a)}}{t^a-t^{-a}}
\in\BZ[t,t^{-1}]
\]
for $m, n\in \BZ$ with $m\geqq n\geqq0$.
We have a natural Hopf algebra structure of $U_q(\Gg)$ given by 
\begin{align}
&\Delta(k_\lambda)=k_\lambda\otimes k_\lambda,
\quad
\Delta(e_i)=e_i\otimes 1+k_i\otimes e_i,
\quad
\Delta(f_i)=f_i\otimes k_i^{-1}+1\otimes f_i,
\\
&
\varepsilon(k_\lambda)=1,
\quad 
\varepsilon(e_i)=\varepsilon(f_i)=0,
\\
&S(k_\lambda)=k_{-\lambda},
\quad
S(e_i)=-k_i^{-1}e_i,
\quad
S(f_i)=-f_ik_i
\end{align}
for 
$\lambda\in\Lambda$, $i\in I$.
We define subalgebras $U_q(\Gh)$, $U_q(\Gn)$, $U_q(\Gb)$, $U_q(\Gn^+)$, $U_q(\Gb^+)$ of $U_q(\Gg)$ by
\begin{gather*}
U_q(\Gh)=\langle k_\lambda\mid\lambda\in\Lambda\rangle,
\\
U_q(\Gn)=\langle f_i\mid i\in I\rangle,
\qquad
U_q(\Gb)=\langle k_\lambda, f_i\mid \lambda\in\Lambda,  i\in I\rangle,
\\
U_q(\Gn^+)=\langle e_i\mid i\in I\rangle,
\qquad
U_q(\Gb^+)=\langle k_\lambda, e_i\mid \lambda\in\Lambda,  i\in I\rangle.
\end{gather*}
We also use 
\[
\tilde{U}_q(\Gn)=S(U_q(\Gn)).
\]
Note that $U_q(\Gh)$, $U_q(\Gb)$ and $U_q(\Gb^+)$ are Hopf subalgebras of $U_q(\Gg)$.
We define the adjoint action of $U_q(\Gg)$ on $U_q(\Gg)$ by 
\[
\ad(x)(y)=\sum_{(x)}x_{(0)}y(Sx_{(1)})
\qquad(x, y\in U_q(\Gg)).
\]
We have
\[
U_q(\Gh)=\bigoplus_{\lambda\in\Lambda}\BF k_\lambda.
\]
For $\lambda\in\Lambda$ we define a character 
\[
\chi_\lambda:U_q(\Gh)\to\BF
\]
by $\chi_\lambda(k_\mu)=q^{(\lambda,\mu)}
=(q^{1/N})^{N(\lambda,\mu)}$ for $\mu\in\Lambda$.
For $\gamma\in Q^+$ we set
\begin{align*}
U_q(\Gn)_{-\gamma}=&
\{x\in U_q(\Gn)\mid
\ad(h)(x)=\chi_{-\gamma}(x)
\;\;(\forall h\in U_q(\Gh))\},
\\
\tilde{U}_q(\Gn)_{-\gamma}=&
\{x\in \tilde{U}_q(\Gn)\mid
\ad(h)(x)=\chi_{-\gamma}(x)
\;\;(\forall h\in U_q(\Gh))\}.
\end{align*}
Then we have
\[
U_q(\Gn)=
\bigoplus_{\gamma\in Q^+}U_q(\Gn)_{-\gamma},
\qquad
\tilde{U}_q(\Gn)=
\bigoplus_{\gamma\in Q^+}\tilde{U}_q(\Gn)_{-\gamma}
\]
with
$
\tilde{U}_q(\Gn)_{-\gamma}
=k_\gamma
U_q(\Gn)_{-\gamma}.
$
The multiplication of $U_q(\Gg)$ induces isomorphisms
\begin{gather*}
U_q(\Gg)\cong U_q(\Gn)\otimes U_q(\Gh)\otimes U_q(\Gn^+),
\\
U_q(\Gb)\cong U_q(\Gh)\otimes U_q(\Gn)
\cong U_q(\Gn)\otimes U_q(\Gh)
\end{gather*}
of $\BF$-modules.
\subsection{}
For a left 
(resp.\ right) $U_q(\Gh)$-module $V$ and $\lambda\in\Lambda$ we set
\begin{gather*}
V_\lambda=\{v\in V\mid hv=\chi_\lambda(h)v
\;(h\in U_q(\Gh))\}
\\
(\text{resp.}\;\;
V_\lambda=\{v\in V\mid vh=\chi_\lambda(h)v
\;(h\in U_q(\Gh))\}).
\end{gather*}
We say that a left (or right) $U_q(\Gh)$-module $V$ is a weight module if $V=\bigoplus_{\lambda\in \Lambda}V_\lambda$.
Let $\Gc=\Gh$, $\Gb$ or $\Gg$.
We say that a left $U_q(\Gc)$-module $V$ is in integrable if 
it is a weight module as a $U_q(\Gh)$-module and 
$\dim U_q(\Gc)v<\infty$ for any $v\in V$.
We define the notion of an integrable right $U_q(\Gc)$-module similarly.
We denote by $\Mod_\inte(U_q(\Gc))$
(resp.\ $\Mod^r_\inte(U_q(\Gc))$) 
the category of integrable left
(resp.\ right) $U_q(\Gc)$-modules.

For $\lambda\in \Lambda^+$ we define a left $U_q(\Gg)$-module $V(\lambda)$ by
\begin{equation}
V(\lambda)=U_q(\Gg)/
\left(
\sum_{h\in U_q(\Gh)}U_q(\Gg)(h-\chi_\lambda(h))
+
\sum_{i\in I}U_q(\Gg)e_i
+
\sum_{i\in I}U_q(\Gg)f_i^{\langle\lambda,\alpha_i^\vee\rangle+1}
\right).
\end{equation}
Then $V(\lambda)$ is a finite dimensional irreducible $U_q(\Gg)$-module.
Moreover, any finite dimensional irreducible $U_q(\Gg)$-module is isomorphic to some $V(\lambda)$.
It is known that a $U_q(\Gg)$-module is integrable if and only if it is a sum of $V(\lambda)$'s for $\lambda\in\Lambda^+$.

For $\lambda\in\Lambda^+$ set
\[
V^*(\lambda)=\Hom_\BF(V(\lambda),\BF).
\]
It is a finite dimensional irreducible right $U_q(\Gg)$-module with respect to the right $U_q(\Gg)$-module structure given by
\[
\langle v^* x,v\rangle=\langle v^*, xv\rangle
\qquad
(v^*\in V^*(\lambda), v\in V(\lambda), x\in U_q(\Gg)).
\]
\subsection{}
Set 
\[
\Ue(\Gh)=\bigoplus_{\lambda\in\Lambda}\BF k_{2\lambda}
\subset 
U_q(\Gh).
\]
We define a twisted action of $W$ on $\Ue(\Gh)$ 
by
\begin{equation}
w\circ k_{2\lambda}=q^{2(w\lambda-\lambda,\rho)}k_{2w\lambda}
\qquad(\lambda\in\Lambda),
\end{equation}
where $\rho\in\Lambda$ is defined by $\langle\rho,\alpha_i^\vee\rangle=1$ for any $i\in I$.

Denote by $Z(U_q(\Gg))$ the center of $U_q(\Gg)$.

The following fact is well-known.
\begin{proposition}
The composite of 
\[
Z(U_q(\Gg))\hookrightarrow
U_q(\Gg)
\cong 
U_q(\Gn)\otimes U_q(\Gh)\otimes U_q(\Gn^+)
\xrightarrow{\varepsilon\otimes1\otimes\varepsilon}
U_q(\Gh)
\]
is an injective algebra homomorphism whose image is
\[
\Ue(\Gh)^{W\circ}
=\{
h\in \Ue(\Gh)\mid w\circ h=h\;(w\in W)\}.
\]
\end{proposition}
Hence we have an isomorphism 
\begin{equation}
\label{eq:HC}
\Xi:Z(U_q(\Gg))\simto \Ue(\Gh)^{W\circ}
\end{equation}
of $\BF$-algebras.
We have
\begin{equation}
\label{eq:z-on-V}
z|_{V(\lambda)}=\chi_\lambda(\Xi(z))\id
\qquad(z\in Z(U_q(\Gg)), \lambda\in\Lambda).
\end{equation}
\subsection{}
Let $\Gc=\Gh$, $\Gb$ or $\Gg$ and 
set $C=H$, $B$, $G$ accordingly.
The dual space $U_q(\Gc)^*=\Hom_\BF(U_q(\Gc),\BF)$ turns out to be a $U_q(\Gc)$-bimodule by
\begin{equation}
\label{eq:bimod}
\langle x\cdot\varphi\cdot y,z\rangle
=\langle\varphi,yzx\rangle
\qquad
(\varphi\in U_q(\Gc)^*, \; x, y, z\in U_q(\Gc)).
\end{equation}
We denote by $O(C_q)$ the subspace of $U_q(\Gc)^*$ consisting of $\varphi\in U_q(\Gc)^*$ satisfying 
$U_q(\Gc)\cdot\varphi\in\Mod_\inte(U_q(\Gc))$.
It is known that we have $\varphi\in O(C_q)$ if and only if $\varphi\cdot U_q(\Gc)\in\Mod^r_\inte(U_q(\Gc))$.
Hence $O(C_q)$ is a $U_q(\Gc)$-subbimodule of $U_q(\Gc)^*$.
Moreover, $O(C_q)$ turns out to be a Hopf algebra whose multiplication, unit, comultiplication, counit, antipode 
are given by the transposes of the comultiplication, counit, multiplication, unit, antipode of $U_q(\Gc)$ respectively.

\begin{remark}
\label{rem:UgOG}
The Hopf algebra $O(C_q)$ is a quantum analogue of the 
coordinate algebra $O(C)$ of the affine algebraic group $C$.
By differentiating the $C$-biaction on $O(C)$:
\[
C\times O(C)\times C\ni(g,\varphi,g')\mapsto 
g\cdot{\varphi}\cdot{g'}\in O(C)
\]
given by $(g\cdot\varphi\cdot{g'})(x)=\varphi(g'xg)$ for $x\in C$
we obtain the corresponding $U(\Gc)$-bimodule structure of
$O(C)$:
\begin{equation}
\label{eq:UgOG}
U(\Gc)\times O(C)\times U(\Gc)
\ni
(u,\varphi,u')\mapsto
u\cdot\varphi\cdot u'\in O(C),
\end{equation}
where $U(\Gc)$ denotes the enveloping algebra of $\Gc$.
The $U_q(\Gc)$-bimodule structure of $O(C_q)$ given by \eqref{eq:bimod} is a quantum analogue of \eqref{eq:UgOG}.
\end{remark}

We denote by $\Comod(O(C_q))$ 
(resp.\ $\Comod^r(O(C_q))$)
the category of left 
(resp.\ right) $O(C_q)$-comodules.
We have equivalences 
\begin{equation}
\label{mod:comodmod}
\Mod_\inte(U_q(\Gc))\cong\Comod^r(O(C_q)),
\qquad
\Mod^r_\inte(U_q(\Gc))\cong\Comod(O(C_q))
\end{equation}
of categories.
The correspondence is given as follows.
Assume that $V$ is a left 
(resp.\ right) $O(C_q)$-comodule 
with respect to $\beta:V\to O(C_q)\otimes V$
(resp.\ $\beta':V\to V\otimes O(C_q)$).
Then the corresponding integrable
right (resp.\ left) $U_q(\Gc)$-module structure of $V$ is given by
\begin{gather*}
\beta(v)=\sum_r\varphi_r\otimes v_r
\;\;
\Longrightarrow\;\;
vu=\sum_r\langle\varphi_r,u\rangle v_r
\\
(\text{resp.}\;
\beta'(v)=\sum_rv_r\otimes\varphi_r\;\;
\Longrightarrow\;\;
uv=\sum_r\langle\varphi_r,u\rangle v_r)
\end{gather*}
for $v\in V$, $u\in U_q(\Gc)$.

We have 
\begin{equation}
O(H_q)=\bigoplus_{\lambda\in\Lambda}\BF\chi_\lambda
\subset U_q(\Gh)^*.
\end{equation}

For $\lambda\in\Lambda$ let
$\tchi_\lambda:U_q(\Gb)\to\BF$ be the character defined by 
\[
\tchi_\lambda(h)=\chi_\lambda(h)
\quad(h\in U_q(\Gh)),
\qquad
\tchi_\lambda(y)=\varepsilon(y)
\quad(y\in U_q(\Gn)).
\]
Then we have
\begin{equation}
\tchi_\lambda\in O(B_q).
\end{equation}
Set 
\[
\tilde{U}_q(\Gn)^\bigstar
=\bigoplus_{\gamma\in Q^+}
(\tilde{U}_q(\Gn)_{-\gamma})^*
\subset 
\tilde{U}_q(\Gn)^*.
\]
Then we have an isomorphism 
\begin{equation}
\tilde{U}_q(\Gn)^\bigstar\otimes O(H_q)
\simto O(B_q)
\end{equation}
of $\BF$-modules sending $\psi\otimes\chi\in \tilde{U}_q(\Gn)^\bigstar\otimes O(H_q)$ to 
$\varphi\in O(B_q)$ given by
\[
\langle\varphi,hy\rangle
=\langle\psi,y\rangle
\langle\chi,h\rangle
\qquad(h\in U_q(\Gh), y\in \tilde{U}_q(\Gn)).
\]

For $\lambda\in\Lambda^+$ we have an embedding 
\[
\Phi_\lambda:V(\lambda)\otimes V^*(\lambda)\to O(G_q)
\]
of $U_q(\Gg)$-bimodules given by 
\[
\langle\Phi_\lambda(v\otimes v^*),x\rangle
=\langle v^*,xv\rangle
\qquad
(v\in V(\lambda), v^*\in V^*(\lambda), x\in U_q(\Gg)).
\]
This gives an isomorphism
\begin{equation}
\label{eq:PW}
O(G_q)\cong \bigoplus_{\lambda\in\Lambda^+}V(\lambda)\otimes V^*(\lambda)
\end{equation}
of $U_q(\Gg)$-bimodules.

We  denote by
\begin{equation}
\res:O(G_q)\to O(B_q)
\end{equation}
the canonical Hopf algebra homomorphism induced by the inclusion $U_q(\Gb)\hookrightarrow U_q(\Gg)$.
\subsection{}
We set
\begin{equation}
\label{eq:Ufdef}
\Uf(\Gg)
=
\{
u\in U_q(\Gg)\mid
\dim\ad(U_q(\Gg))(u)<\infty\}.
\end{equation}
It is a subalgebra of $U_q(\Gg)$ satisfying 
\begin{equation}
\label{eq:Uf}
\Delta(\Uf(\Gg))\subset
U_q(\Gg)\otimes \Uf(\Gg).
\end{equation}
Denote by $\Ue(\Gg)$ the subalgebra of $U_q(\Gg)$ generated by $k_{2\lambda}$ \;($\lambda\in\Lambda$), $\tilde{U}_q(\Gn)$, $U_q(\Gn^+)$.
Then we have
\begin{equation}
\label{eq:Ue}
\Delta(\Ue(\Gg))\subset
U_q(\Gg)\otimes \Ue(\Gg),
\end{equation}
and the multiplication of $\Ue(\Gg)$ induces an isomorphism
\begin{equation}
\Ue(\Gg)\cong
\tilde{U}_q(\Gn)\otimes
\Ue(\Gh)\otimes 
U_q(\Gn^+)
\end{equation}
of $\BF$-modules.
It is known that $k_{2\lambda}$ for $\lambda\in\Lambda^+$ is contained in $\Uf(\Gg)$, and the multiplicative set $T=\{k_{2\lambda}\mid\lambda\in\Lambda^+\}$ satisfies the left and right Ore conditions in $\Uf(\Gg)$.
Moreover, we have $T^{-1}\Uf(\Gg)=\Ue(\Gg)$
(see Joseph \cite{JoB}).
In particular, we have the following.
\begin{lemma}
\label{lem:gen}
The linear map
$U_q(\Gh)\otimes \Uf(\Gg)\to U_q(\Gg)$
given by the multiplication of $U_q(\Gg)$ is surjective.
\end{lemma}

\section{$\CO$-modules in the $\Proj$ picture}
\subsection{}
We briefly recall the $\Proj$-construction of the category of $\CO$-modules on the quantized flag manifolds in the following
(see Lunts and Rosenberg \cite{LR}).

Set 
\begin{align}
\label{eq:A}
{A_q}=&
\{\varphi\in O(G_q)\mid
\varphi\cdot y=
\varepsilon(y)\varphi
\quad(y\in U_q(\Gn))\},
\\
\label{eq:Alam1}
{A_q}(\lambda)=&
\{\varphi\in O(G_q)\mid
\varphi\cdot y=
\tchi_\lambda(y)\varphi
\quad(y\in U_q(\Gb)\}
\\
\nonumber
=&
\{\varphi\in O(G_q)\mid
(\res\otimes1)(\Delta(\varphi))=
\tchi_\lambda\otimes \varphi)\}
\qquad(\lambda\in\Lambda).
\end{align}
They are $(U_q(\Gg), U_q(\Gh))$-submodules of 
the $U_q(\Gg)$-bimodule $O(G_q)$.
By \eqref{eq:PW}  we have
\[
{A_q}=\bigoplus_{\lambda\in\Lambda^+}{A_q}(\lambda),
\]
and
\begin{equation}
\label{eq:Alam2}
{A_q}(\lambda)
\cong
\begin{cases}
 V(\lambda)\otimes V^*(\lambda)_\lambda
\quad&(\lambda\in\Lambda^+)
\\
0&(\lambda\notin\Lambda^+),
\end{cases}
\end{equation}
as a $(U_q(\Gg), U_q(\Gh))$-module.
It is easily seen that  ${A_q}(0)=\BF 1$ and ${A_q}(\lambda){A_q}(\mu)\subset {A_q}(\lambda+\mu)$ for $\lambda, \mu\in\Lambda^+$ so that
 ${A_q}$ turns out to be a $\Lambda$-graded $\BF$-algebra.
 
By Joseph \cite{Jo0} we have the following.
\begin{proposition}
\label{prop:A}
\begin{itemize}
\item[(i)]
If $\varphi, \psi\in {A_q}$ satisfy $\varphi\psi=0$, then 
we have $\varphi=0$ or $\psi=0$.
\item[(ii)]
The $\BF$-algebra ${A_q}$ is left and right noetherian.
\end{itemize}
\end{proposition}
We define an abelian category 
$\Mod(\CO_{\CB_q})$ by
\begin{equation}
\Mod(\CO_{\CB_q}):=
\Mod_\Lambda({A_q})/\Tor_{\Lambda^+}({A_q})
:=
\Sigma^{-1}\Mod_\Lambda({A_q}),
\end{equation}
where $\Tor_{\Lambda^+}({A_q})$ is the full subcategory of 
$\Mod_\Lambda({A_q})$
consisting of $K\in\Mod_\Lambda({A_q})$ 
such that for any $k\in K$ 
we have ${A_q}(\lambda)k=0$ for sufficiently large
$\lambda\in\Lambda^+$, i.e., 
there exists some $n>0$ satisfying 
\[
\lambda\in \Lambda^+,
\quad
\langle\lambda,\alpha_i^\vee\rangle\geqq n
\quad(\forall i\in I)
\quad\Rightarrow\quad
{A_q}(\lambda)k=0.
\]
Moreover, $\Sigma$ consists of morphisms in $\Mod_\Lambda({A_q})$ whose kernel and cokernel belong to $\Tor_{\Lambda^+}({A_q})$, 
and 
$\Sigma^{-1}\Mod_\Lambda({A_q})$ denotes the localization of the category $\Mod_\Lambda({A_q})$ so that the morphisms in $\Sigma$ turn out to be isomorphisms (see \cite{GZ}, \cite{P} for the notion of the localization of categories).
For $\GM, \GN\in\Mod(\CO_{\CB_q})$ we denote by 
$\Hom_{\CO_{\CB_q}}(\GM,\GN)$ the set of morphisms from $\GM$ to $\GN$.

Denote by 
\begin{equation}
\label{eq:pi}
\varpi:
\Mod_\Lambda({A_q})
\to
\Mod(\CO_{\CB_q})
\end{equation}
the natural functor.
It admits a right adjoint 
\begin{equation}
\label{eq:Gamma*}
\tGamma:
\Mod(\CO_{\CB_q})
\to
\Mod_\Lambda({A_q})
\end{equation}
(see \cite[Ch 4, Proposition 5.2]{P}).
Note that $\varpi$ is exact and $\tGamma$ is left exact.
We have canonical morphisms
\begin{align}
\label{eq:canon}
&M\to \tGamma(\varpi M)
\qquad(M\in\Mod_\Lambda({A_q})),
\\
\label{eq:canon2}
&\varpi(\tGamma(\GM))\to \GM
\qquad(\GM\in\Mod(\CO_{\CB_q}))
\end{align}
corresponding to 
\begin{align*}
&\id\in\Hom_{\CO_{\CB_q}}(\varpi M,\varpi M)
\cong
\Hom^\gr_{{A_q}}(M,\tGamma(\varpi M)),
\\
&\id\in\Hom\Hom^\gr_{{A_q}}(\tGamma(\GM),\tGamma(\GM))
\cong
\Hom_{\CO_{\CB_q}}(\varpi(\tGamma(\GM)),\GM)
\end{align*}
respectively.
By \cite[Ch 4, Proposition 4.3]{P} \eqref{eq:canon2} is an isomorphism so that we have an equivalence
\begin{equation}
\label{eq:piGamma}
\varpi\circ\tGamma\simto\Id
\end{equation}
of functors.

The functor 
\[
(\bullet)[\lambda]:
\Mod_\Lambda({A_q})
\to
\Mod_\Lambda({A_q})
\qquad(\lambda\in\Lambda)
\]
(see \ref{subsec:notation}) induces the functor
\begin{equation}
\label{eq:translation}
(\bullet)[\lambda]:
\Mod(\CO_{\CB_q})
\to
\Mod(\CO_{\CB_q})
\qquad(\lambda\in\Lambda).
\end{equation}
We have
$(\bullet)[0]=\Id$ and 
$(\bullet)[\lambda][\mu]=(\bullet)[\lambda+\mu]$ for $\lambda, \mu\in \Lambda$.
It is easily seen that 
\[
\varpi(M[\lambda])=(\varpi M)[\lambda],
\qquad
\tGamma(\GM[\lambda])
=(\tGamma(\GM))[\lambda]
\]
for $M\in\Mod_\Lambda({A_q})$, $\GM\in\Mod(\CO_{\CB_q})$, $\lambda\in\Lambda$.
\begin{remark}
\label{rem:Proj1}
Similarly to ${A_q}$ we can also define at $q=1$ a $\Lambda$-graded $\BC$-algebra $A$ using $U(\Gg)$ and $O(G)$ instead of $U_q(\Gg)$ and $O(G_q)$.
It is a subalgebra of the coordinate algebra $O(G)$ of $G$, and is identified with the homogeneous coordinate algebra of the flag manifold $\CB=B\backslash G$ regarded as a projective scheme. 
Namely, we have $\CB=\Proj(A)$, and the category $\Mod(\CO_\CB)$ of 
quasi-coherent $\CO_\CB$-modules is given by 
\[
\Mod(\CO_{\CB})=
\Mod_\Lambda(A)/\Tor_{\Lambda^+}(A).
\]
The functor corresponding to \eqref{eq:translation} is given by 
\[
\Mod(\CO_{\CB})\to \Mod(\CO_{\CB})
\qquad
(\GM\mapsto
\GM\otimes_{\CO_\CB} \CO_\CB(\lambda)),
\]
where 
$\CO_\CB(\lambda)$ is the invertible 
$\CO_\CB$-module corresponding to 
$\lambda\in\Lambda$.
Moreover, the functor corresponding to \eqref{eq:Gamma*} is given by 
\[
\Mod(\CO_{\CB})\to \Mod_\Lambda(A)
\qquad
(\GM\mapsto
\bigoplus_{\lambda\in\Lambda}
\Gamma(\CB,\GM\otimes_{\CO_\CB}\CO_\CB(\lambda))).
\]
\end{remark}
\subsection{}
Note that for $\lambda\in\Lambda^+$ and $w\in W$ we have $\dim {A_q}(\lambda)_{w^{-1}\lambda}=1$ by
\eqref{eq:Alam2}, where ${A_q}(\lambda)$ is regarded as a $U_q(\Gh)$-module via the left action of $U_q(\Gg)$ on ${A_q}$. 
For each $\lambda\in\Lambda^+$ take a non-zero element $c_{w,\lambda}\in {A_q}(\lambda)_{w^{-1}\lambda}$.
We may assume  
$c_{w,\lambda}c_{w,\mu}=c_{w,\lambda+\mu}$ for any $\lambda, \mu\in \Lambda^+$.
Set
\begin{equation}
\label{eq:Sw}
S_w=\{c_{w,\lambda}\mid
\lambda\in\Lambda^+\}\subset {A_q}
\qquad(w\in W).
\end{equation}
By Joseph \cite{Jo0} 
the multiplicative subset $S_w$ of $A_q$ satisfies the following (see also \cite{T0}).
\begin{proposition}
\label{prop:Ore}
Let $w\in W$.
\begin{itemize}
\item[(i)]
$S_w$ satisfies the left and right Ore conditions in ${A_q}$.
\item[(ii)]
The canonical homomorphism ${A_q}\to S_w^{-1}{A_q}$ is injective.
\end{itemize}
\end{proposition}
Since $S_w$ consists of homogeneous elements, the localization $S_w^{-1}{A_q}$ turns out to be a $\Lambda$-graded $\BF$-algebra.
Set
\[
R_{w,q}=(S_w^{-1}{A_q})(0)
\qquad(w\in W).
\]
For $\lambda\in\Lambda$ write $\lambda=\lambda_1-\lambda_2$ for $\lambda_1, \lambda_2\in\Lambda^+$ and set $c_{w,\lambda}=c_{w,\lambda_1}
c_{w,\lambda_2}^{-1}\in (S_w^{-1}{A_q})(\lambda)$.
It does not depend on the choice of $\lambda_1, \lambda_2$.
Then the  homogeneous component $(S_w^{-1}{A_q})(\lambda)$ is a free $R_{w,q}$-module of rank one
generated by $c_{w,\lambda}$.
Hence
we have an equivalence 
\[
\Mod_\Lambda(S_w^{-1}{A_q})\cong
\Mod(R_{w,q})
\qquad(M\mapsto M(0))
\]
of abelian categories.

We see easily that $S_w^{-1}M=0$ for any $M\in \Tor_{\Lambda^+}({A_q})$, and hence the localization functor $\Mod_\Lambda({A_q})\to \Mod_\Lambda(S_w^{-1}{A_q})$ induces an exact functor 
\begin{equation}
\label{eq:jw}
j_{w,q}^*:\Mod(\CO_{\CB_q})\to \Mod(R_{w,q}).
\end{equation}
\begin{remark}
Similarly to $R_{w,q}$ and $j_{w,q}^*$, we can also define at $q=1$ a commutative $\BC$-algebra $R_{w}$ and an exact functor $j_{w}^*:\Mod(\CO_\CB)\to \Mod(R_{w})$.
Then $R_{w}$ is naturally isomorphic to the coordinate algebra of the affine open subset $U_w=B\backslash BB^+{w}$ of $\CB$ so that 
$\Mod(R_{w})$ is identified with the category $\Mod(\CO_{U_w})$ of quasi-coherent $\CO_{U_w}$-modules. 
Moreover, $j^*_{w}:\Mod(\CO_{\CB})\to \Mod(\CO_{U_w})$ is the pull-back with respect to the embedding $j_{w}:U_w\to \CB$.
\end{remark}

It is shown in \cite{Jo0} and \cite{LR} using reduction to $q=1$ that for any $\mu\in\Lambda^+$ we have 
\begin{equation}
\label{eq:Apatch}
\sum_{w\in W}{A_q}(\lambda){A_q}(\mu)_{w^{-1}\mu}
={A_q}(\lambda+\mu)
\end{equation}
for any sufficiently large $\lambda\in\Lambda^+$.
From this we obtain the following patching property.
\begin{proposition}[\cite{LR}]
\label{prop:patch}
A morphism $f$ in $\Mod(\CO_{\CB_q})$ is an isomorphism if $j_w^*f$ is an isomorphism for any $w\in W$.
\end{proposition}

Lunts and Rosenberg \cite{LR} proved the following using Proposition \ref{prop:patch} among other things (see also \ref{subsec:Ind} below).
\begin{proposition}
\label{prop:BW}
We have $\tGamma(\varpi {A_q})\cong {A_q}$.
\end{proposition}
\section{$\CO$-modules in the pull-back picture}
\subsection{}
In this section we recall another approach to  the quantized flag manifolds given by 
Backelin and Kremnizer \cite{BK}.

We define an abelian category  
$\Mod(\CO_{G_q},B_q)$ as follows.
An object is a left  $O(G_q)$-module $M$ equipped with a left $O(B_q)$-comodule structure 
\[
\beta_M:M\to O(B_q)\otimes M
\]
such that $\beta_M$ is a homomorphism of left $O(G_q)$-modules with respect to the diagonal left action of $O(G_q)$ on 
$O(B_q)\otimes M$ 
(equivalently, the left $O(G_q)$-module structure 
$O(G_q)\otimes M\to M$ is 
a homomorphism of left $O(B_q)$-comodules 
with respect to the diagonal left coaction of 
$O(B_q)$ on $O(G_q)\otimes M$).
Here, $O(G_q)$ is regarded as a left $O(B_q)$-comodule by 
$(\res\otimes1)\circ\Delta:
O(G_q)\to O(B_q)\otimes O(G_q)$.
A morphism  in 
$\Mod(\CO_{G_q},B_q)$ is a homomorphism of left $O(G_q)$-modules respecting the left coaction of $O(B_q)$.

\begin{remark}
At $q=1$ we can similarly define a category 
$\Mod(\CO_{G},B)$
using $O(G)$, $O(B)$ instead of $O(G_q)$, $O(B_q)$.
Namely, an object of 
$\Mod(\CO_{G},B)$
is an $O(G)$-module $M$ equipped with a left $O(B)$-comodule structure 
satisfying the obvious compatibility condition.
This category
is naturally identified with
the category of $B$-equivariant quasi-coherent $\CO_G$-modules.
Moreover, we have an equivalence 
\begin{equation}
\label{eq:O-equivalence}
\Mod(\CO_{\CB})\cong
\Mod(\CO_{G},B)
\qquad
(\GM\mapsto p^*\GM)
\end{equation}
of categories, where $p:G\to \CB=B\backslash G$ is the canonical morphism.
\end{remark}
\subsection{}
For $\lambda\in\Lambda$ we define a one-dimensional left $O(B_q)$-comodule $\BF_\lambda=\BF 1_\lambda$ by
\[
\BF_\lambda\to O(B_q)\otimes\BF_\lambda
\qquad(1_\lambda\mapsto \tchi_\lambda\otimes 1_\lambda),
\]
and define a functor 
\begin{equation}
\label{eq:stranslation}
(\bullet)[\lambda]:
\Mod(\CO_{G_q},B_q)
\to 
\Mod(\CO_{G_q},B_q)
\qquad(M\mapsto M[\lambda])
\end{equation}
by $M[\lambda]=M\otimes \BF_{-\lambda}$, 
where the $O(G_q)$-module structure of $M[\lambda]$ is given by the $O(G_q)$-action on the first factor and the $O(B_q)$-comodule structure is given by the diagonal coaction.
We have
$(\bullet)[0]=\Id$ and 
$(\bullet)[\lambda][\mu]=(\bullet)[\lambda+\mu]$.

We define a left exact functor 
\begin{equation}
\label{eq:sGamma*}
(\bullet)^{N_q}:\Mod(\CO_{G_q},B_q)\to
\Mod_\Lambda({A_q})
\end{equation}
by
\[
M^{N_q}=
\sum_{\lambda\in\Lambda}
M^{N_q}(\lambda)
\subset M,
\quad
M^{N_q}(\lambda)
=
\{m\in M\mid\beta_M(m)=\tchi_\lambda\otimes m\}.
\]
Note that the sum in the first formula is a direct sum.
We define 
\begin{equation}
\label{eq:sGamma}
(\bullet)^{B_q}:\Mod(\CO_{G_q},B_q)\to
\Mod(\BF)
\end{equation}
by
$M^{B_q}=
M^{N_q}(0)
$.
\section{Comparison of the two categories of $\CO$-modules}
Let $K\in\Mod_\Lambda({A_q})$.
Then $O(G_q)\otimes_{{A_q}}K$ turns out to be an object of $\Mod(\CO_{G_q},B_q)$ by the left $O(G_q)$-module structure given the left multiplication on the first factor and the left $O(B_q)$-comodule structure given by 
\[
\beta_{O(G_q)\otimes_{{A_q}}K}(\varphi\otimes k)=
\sum_{(\varphi)}
\res(\varphi_{(0)})\tchi_\lambda\otimes
(\varphi_{(1)}\otimes k)
\qquad
(\varphi\in O(G_q),\; k\in K(\lambda)).
\]
This defines a functor 
\begin{equation}
\label{eq:spi}
O(G_q)\otimes_{A_q}(\bullet):
\Mod_\Lambda({A_q})
\to
\Mod(\CO_{G_q},B_q).
\end{equation}
By \cite{BK} we have the following result.
\begin{proposition}
\label{prop:AZBK}
There exists an equivalence 
\begin{equation}
\label{eq:AZBKsim}
p_q^*:\Mod(\CO_{\CB_q})
\simto
\Mod(\CO_{G_q},B_q)
\end{equation}
of abelian categories such that the following diagrams are commutative:
\begin{equation}
\label{eq:AZBK1}
\vcenter{
\xymatrix{
   &
   \Mod_\Lambda({A_q})
\ar[dl]_\varpi
\ar[dr]^{\;O(G_q)\otimes_{A_q}(\bullet)}
   \\
   \Mod(\CO_{\CB_q})
\ar[rr]^{p_q^*}_\cong
&
&
\Mod(\CO_{G_q},B_q),
}}
\end{equation}
\begin{equation}
\label{eq:AZBK2}
\vcenter{
\xymatrix{
\Mod(\CO_{\CB_q})
\ar[d]_{(\bullet)[\lambda]}
\ar[r]^-{p_q^*}_-\cong
&
\Mod(\CO_{G_q},B_q)
\ar[d]^{(\bullet)[\lambda]}
\\
\Mod(\CO_{\CB_q})
\ar[r]^-{p_q^*}_-\cong
&
\Mod(\CO_{G_q},B_q)
}
}
\qquad(\lambda\in\Lambda).
\end{equation}
\end{proposition}

Proposition \ref{prop:AZBK} is a consequence of (an obvious generalization of) \cite[Theorem 4.5]{AZ} 
giving a characterization of the category of $\CO$-modules for non-commutative projective schemes.
In order to show that the assumption of \cite[Theorem 4.5]{AZ} is satisfied in our situation one needs to use 
some facts on the derived induction functor for quantum groups (\cite[Corollary 5.7, Theorem 5.8]{APW}) and 
the fact that 
 $O(G_q)$ is a left noetherian ring (\cite{JoB}).

\begin{corollary}
\label{cor:equivalence}
\begin{itemize}
\item[(i)]
The functor 
\begin{equation}
\label{eq:AZBK-inv}
\varpi\circ((\bullet)^{N_q})
:
\Mod(\CO_{G_q},B_q)
\to
\Mod(\CO_{\CB_q})
\end{equation}
gives an inverse to \eqref{eq:AZBKsim}.

\item[(ii)]
The following diagram is commutative:
\begin{equation}
\vcenter{
\xymatrix{
   \Mod(\CO_{\CB_q})
\ar[rr]^-{p_q^*}_-\cong
\ar[dr]_{\tGamma}
&
&
\Mod(\CO_{G_q},B_q)
\ar[dl]^{(\bullet)^{N_q}}
\\
   &
   \Mod_\Lambda({A_q}).
}}
\end{equation}
\item[(iii)]
The functor 
\eqref{eq:spi} is exact.
\item[(iv)]
For $M\in\Tor_{\Lambda^+}({A_q})$ we have
$O(G_q)\otimes_{A_q}M=0$.
\item[(v)]
We have
\[
O(G_q)\otimes_{A_q}M^{N_q}\cong M
\qquad(M\in\Mod(\CO_{G_q},B_q)).
\]
\end{itemize}
\end{corollary}
\begin{proof}
(ii) follows from the fact  that \eqref{eq:sGamma*}
is right adjoint to \eqref{eq:spi}.
In view of (ii) the assertions (iii), (iv), (v) are
consequences of 
the exactitude of $\varpi$, the definition of $\varpi$,
and \eqref{eq:piGamma} respectively.
Hence have
\[
p_q^*\circ\varpi\circ((\bullet)^{N_q})
=
(O(G_q)\otimes_{A_q}(\bullet))\circ((\bullet)^{N_q})
=\Id,
\]
and (i) follows.
\end{proof}
\section{Equivariant $\CO$-modules}
\subsection{}
We denote by $\Mod^\eq_\Lambda({A_q})$ the category whose objects are 
$K\in \Mod_\Lambda({A_q})$ equipped with a right $O(G_q)$-comodule structure 
$\gamma_K:K\to K\otimes O(G_q)$ satisfying
\begin{itemize}
\item[(a)]
$\gamma_K(K(\lambda))\subset K(\lambda)\otimes O(G_q)$ \qquad($\lambda\in\Lambda$),
\item[(b)]
$\gamma_K(\varphi k)=\Delta(\varphi)\gamma_K(k)$
\qquad($\varphi\in {A_q}, k\in K$).
\end{itemize}
Note that (b) makes sense by $\Delta({A_q})\subset {A_q}\otimes O(G_q)$.
In terms of the corresponding left integrable $U_q(\Gg)$-module structure of $K$
the condition (b) is equivalent to 
\begin{itemize}
\item[(b)']
$u(\varphi k)=
\sum_{(u)}
(u_{(0)}\cdot\varphi)
(u_{(1)}k)$
\qquad($u\in U_q(\Gg), \varphi\in {A_q}, k\in K$).
\end{itemize}

We define a category $\Mod^\eq(\CO_{\CB_q})$ by
\begin{equation}
\Mod^\eq(\CO_{\CB_q})
=\Mod^\eq_\Lambda({A_q})/
(\Mod^\eq_\Lambda({A_q})\cap\Tor_{\Lambda^+}({A_q})),
\end{equation}
where $\Mod^\eq_\Lambda({A_q})\cap\Tor_{\Lambda^+}({A_q})$ denotes the subcategory of $\Mod^\eq_\Lambda({A_q})$ consisting of $M\in \Mod^\eq_\Lambda({A_q})$ which belongs to $\Tor_{\Lambda^+}({A_q})$ as an object of 
$\Mod_\Lambda({A_q})$.
By abuse of notation we denote by
\begin{equation}
\varpi:\Mod^\eq_\Lambda({A_q})\to
\Mod^\eq(\CO_{\CB_q})
\end{equation}
the natural functor.

We denote by $\Mod^\eq(\CO_{G_q},B_q)$ the category whose objects are $M\in\Mod(\CO_{G_q},B_q)$ equipped with a 
right $O(G_q)$-comodule structure 
$\gamma_M:M\to M\otimes O(G_q)$ 
such that $\beta_M$ and $\gamma_M$ commute with each other and that $\gamma_M$ is a homomorphism of left $O(G_q)$-modules with respect to the diagonal left action of $O(G_q)$ on $M\otimes O(G_q)$.

\begin{remark}
At $q=1$ the category $\Mod^\eq(\CO_\CB)$ 
defined similarly to $\Mod^\eq(\CO_{\CB_q})$  is naturally identified with the category of 
$G$-equivariant quasi-coherent $\CO_\CB$-modules with respect to the right $G$-action on $\CB=B\backslash G$  given by
\[
\CB\times G\ni(Bx,g)\mapsto Bxg\in\CB.
\]
Moreover, the category $\Mod^\eq(\CO_G,B)$ 
defined similarly to $\Mod^\eq(\CO_{G_q},B_q)$  is naturally identified with the category of 
quasi-coherent $\CO_G$-modules 
equivariant 
with respect to the $(B,G)$-biaction on $G$  given by
\[
B\times G\times G\ni(b,x,g)\mapsto bxg\in G.
\]
We have an equivalence
\begin{equation}
\label{eq:1-O-eq}
\Mod^\eq(\CO_\CB)
\cong
\Mod^\eq(\CO_G,B)
\end{equation}
induced by 
\eqref{eq:O-equivalence}
\end{remark}

The functors \eqref{eq:sGamma*}, 
\eqref{eq:sGamma}
 and \eqref{eq:spi} induce
\begin{align}
\label{eq:sGamma*eq}
(\bullet)^{N_q}:&\Mod^\eq(\CO_{G_q},B_q)\to
\Mod^\eq_\Lambda({A_q}), 
\\
\label{eq:sGammaeq}
(\bullet)^{B_q}:&\Mod^\eq(\CO_{G_q},B_q)\to
\Comod^r(O(G_q)), 
\\
\label{eq:spieq}
O(G_q)\otimes_{A_q}(\bullet):&
\Mod^\eq_\Lambda({A_q})
\to
\Mod^\eq(\CO_{G_q},B_q).
\end{align}
Here, for $M\in \Mod^\eq(\CO_{G_q},B_q)$ the right $O(G_q)$-comodule structure of 
$M^{N_q}$ is given by the restriction of $\gamma_M$, and
for $K\in\Mod^\eq_\Lambda({A_q})$ the right $O(G_q)$-comodule structure of 
$O(G_q)\otimes_{A_q}K$ is given by 
\[
\gamma_K(k)=\sum_rk_r\otimes\psi_r\;
\Longrightarrow\;
\gamma_{O(G_q)\otimes_{A_q}K}(\varphi\otimes k)=
(\varphi_{(0)}\otimes k_r)\otimes\varphi_{(1)}\psi_r.
\]
\begin{proposition}
\label{prop:AZBKeq}
The functor \eqref{eq:AZBKsim} induces the equivalence 
\begin{equation}
\label{eq:AZBKeqsim}
p_q^*:
\Mod^\eq(\CO_{\CB_q})
\simto
\Mod^\eq(\CO_{G_q},B_q)
\end{equation}
of abelian categories.
\end{proposition}
\begin{proof}
By Corollary \ref{cor:equivalence} (iv) and the universality of the localization of categories 
we obtain a functor 
$p_q^*:
\Mod^\eq(\CO_{\CB_q})
\to
\Mod^\eq(\CO_{G_q},B_q)
$ induced by \eqref{eq:AZBKsim}.
It is easily seen from Corollary \ref{cor:equivalence} (i)  that its inverse is given by
$
\varpi\circ((\bullet)^{N_q})
:
\Mod^\eq(\CO_{G_q},B_q)
\to
\Mod^\eq(\CO_{\CB_q})
$.
\end{proof}
We have the following commutative diagram:
\begin{equation}
\vcenter{
\xymatrix{
   & 
   \Mod^\eq_\Lambda({A_q})
\ar[dl]_\varpi
\ar[dr]^{\;O(G_q)\otimes_{A_q}(\bullet)}
   \\
   \Mod^\eq(\CO_{\CB_q})
\ar[rr]^-{p_q^*}_-\cong
&
&
\Mod^\eq(\CO_{G_q},B_q).
}}
\end{equation}
We define a functor
\begin{equation}
\tGamma:
\Mod^\eq(\CO_{\CB_q})
\to
\Mod_\Lambda^\eq({A_q})
\end{equation}
by the commutative diagram:
\begin{equation}
\vcenter{
\xymatrix{
   \Mod^\eq(\CO_{\CB_q})
\ar[rr]^-{p_q^*}_-\cong
\ar[dr]_{\tGamma}
&
&
\Mod^\eq(\CO_{G_q},B_q)
\ar[dl]^{(\bullet)^{N_q}}
\\
   &
   \Mod^\eq_\Lambda({A_q}).
}}
\end{equation}
Then we have the commutative diagram:
\begin{equation}
\vcenter{
\xymatrix{
\Mod^\eq_\Lambda({A_q})
\ar[r]
\ar[d]_\tGamma
&
\Mod_\Lambda({A_q})
\ar[d]^\tGamma
\\
\Mod^\eq(\CO_{\CB_q})
\ar[r]
&
\Mod(\CO_{\CB_q}),
}}
\end{equation}
where the upper horizontal arrow is a forgetful functor and the lower horizontal arrow is the functor induced by it.
\subsection{}
We define a functor
\begin{align}
\label{eq:fiber1}
\prescript{G_q}{}{(}\bullet):\Mod^\eq(\CO_{G_q},B_q)\to \Comod(O(B_q))
\end{align}
by
\[
\prescript{G_q}{}{M}=\{m\in M\mid \gamma_M(m)=m\otimes1\}.
\]
We have also a functor 
\begin{align}
\label{eq:Upsilon}
O(G_q)\otimes(\bullet):
\Comod(O(B_q))\to
\Mod^\eq(\CO_{G_q},B_q)
\end{align}
in the inverse direction 
sending $L\in\Comod(O_q	(G))$ to 
$O(G_q)\otimes L$, 
where $O(G_q)\otimes L$ is regarded as an object of 
$\Mod^\eq(\CO_{G_q},B_q)$ via 
the left $O(G_q)$-module structure given by the left multiplication on the first factor
and the left $O(B_q)$-comodule structure given by the diagonal coaction on the tensor product.

By \cite[Section 3.5]{BK} we have the following.
\begin{proposition}
\label{prop:xi}
The functor \eqref{eq:fiber1} gives an equivalence of categories.
Its inverse is given by \eqref{eq:Upsilon}.
\end{proposition}
\subsection{}
\label{subsec:Ind}
We define a functor
\[
(\bullet)^\dagger:
\Comod^r(O(B_q))
\to
\Comod(O(B_q))
\qquad
(M\mapsto M^\dagger)
\]
as follows.
As an $\BF$-module we have
$M\cong M^\dagger$ 
($m\leftrightarrow m^\dagger$).
The left $O(B_q)$-comodule structure 
$\beta^\dagger:M^\dagger\to O(B_q)\otimes M^\dagger$ 
of $M^\dagger$ is 
defined from
the right $O(B_q)$-comodule structure 
$\beta:M\to M\otimes O(B_q)$ 
of $M$ by 
\[
\beta(m)=\sum_rm_r\otimes\psi_r
\;\;
\Longrightarrow
\;\;
\beta^\dagger(m^\dagger)
=
\sum_r S^{-1}\psi_r\otimes m_r^\dagger.
\]
It is easily seen that the composite of
\begin{align*}
\Comod^r(O(B_q))
&
\xrightarrow{(\bullet)^\dagger}
\Comod(O(B_q))
\xrightarrow{O(G_q)\otimes(\bullet)}
\Mod^\eq(\CO_{G_q},B_q)
\\&
\xrightarrow{(\bullet)^{B_q}}
\Comod^r(O(G_q))
\end{align*}
is nothing but the induction functor
\[
\Ind:
\Comod^r(O(B_q))
\to
\Comod^r(O(G_q))
\]
investigated in \cite{APW}.

Denote by $\BF'_\lambda=\BF 1'_\lambda$ 
the one-dimensional right $O(B_q)$-comodule given by
\[
\BF'_\lambda\to\BF'_\lambda\otimes O(B_q)
\qquad(1'_\lambda\mapsto 1'_\lambda\otimes \tchi_\lambda).
\]
Proposition \ref{prop:BW} also follows from the well-known fact
\[
\Ind(\BF'_\lambda)
=
\begin{cases}
V(\lambda)
\qquad&(\lambda\in\Lambda^+)
\\
0&(\lambda\notin\Lambda^+)
\end{cases}
\]
for $\lambda\in\Lambda$
in view of Proposition \ref{prop:AZBKeq}.

\subsection{}
Let $K\in\Mod^\eq_\Lambda({A_q})$.
Then $\BF\otimes_{A_q}K$ with respect to the restriction $\varepsilon|_{A_q}:{A_q}\to\BF$ of the counit $\varepsilon$ of $O(G_q)$ is naturally endowed with a left $O(B_q)$-comodule structure induced by that of $K$.
This defines a functor
\begin{equation}
\label{eq:eq-inv}
\BF\otimes_{A_q}(\bullet):
\Mod^\eq_\Lambda({A_q})\to\Comod(O(B_q)).
\end{equation}
\begin{lemma}
\label{lem:equ-fiber}
For $K\in\Mod^\eq_\Lambda({A_q})$ 
we have a canonical isomorphism
\[
\prescript{G_q}{}{(}O(G_q)\otimes_{A_q}K)
\cong\BF\otimes_{A_q}K
\]
in $\Comod(O(B_q))$.
\end{lemma}
\begin{proof}
By Proposition \ref{prop:xi} we can canonically associate to $K$ an object $L$ of $\Comod(O(B_q))$ such that 
$O(G_q)\otimes_{A_q}K\cong O(G_q)\otimes L$.
Then we have canonical isomorphisms
\begin{align*}
&
\prescript{G_q}{}{(}O(G_q)\otimes_{A_q}K)
\cong
\prescript{G_q}{}(O(G_q)\otimes L)\cong L,
\\
&
\BF\otimes_{A_q}K
\cong
\BF\otimes_{O(G_q)}(O(G_q)\otimes_{A_q}K)
\cong
\BF\otimes_{O(G_q)}(O(G_q)\otimes L)
\cong
L.
\end{align*}
\end{proof}
By Proposition \ref{prop:AZBKeq}, 
Proposition \ref{prop:xi}, 
Lemma \ref{lem:equ-fiber} 
we have the following.
\begin{proposition}
\label{prop:iota}
The functor \eqref{eq:eq-inv} induces an equivalence
\begin{equation}
\label{eq:iota}
\iota_q^*:
\Mod^\eq(\CO_{\CB_q})\simto \Comod(O(B_q)).
\end{equation}
Moreover, we have the following commutative diagram:
\begin{equation}
\vcenter{
\xymatrix{
   \Mod^\eq(\CO_{\CB_q})
\ar[rr]_\cong^{p_q^*}
\ar[dr]^{\iota_q^*}_\cong
&
&
\Mod^\eq(\CO_{G_q},B_q)
\ar[dl]_{\prescript{G_q}{}{(}\bullet)}^\cong
\\
   &
  \Comod(O(B_q)).
}}
\end{equation}
\end{proposition}
In particular, \eqref{eq:eq-inv} is an exact functor 
by $\BF\otimes_{A_q}(\bullet)=\iota_q^*\circ\varpi$.\begin{remark}
At $q=1$ we have a functor 
\[
\prescript{G}{}{(}\bullet):
\Mod^\eq(\CO_{G},B)\simto 
\Comod(O(B))
\]
defined similarly to \eqref{eq:fiber1}.
This corresponds under 
the equivalence \eqref{eq:1-O-eq} 
to the inverse image functor 
\[
\iota^*:\Mod^\eq(\CO_\CB)
\simto\Comod(O(B))
\]
for the embedding $\iota:\{Be\}\hookrightarrow\CB=B\backslash G$.
Our $\iota_q^*$ above is a quantum analogue of $\iota^*$.
\end{remark}

\section{$\DD$-modules in the $\Proj$ picture}
\subsection{}
Recall that ${A_q}$ is a $(U_q(\Gg),U_q(\Gh))$-submodule of the $U_q(\Gg)$-bimodule $O(G_q)$.
For $\varphi\in {A_q}$, $u\in U_q(\Gg)$, $h\in U_q(\Gh)$ we define 
$\ell_\varphi, \deru_u, \sigma_h\in \End_\BF({A_q})$ by
\[
\ell_\varphi(\psi)=\varphi\psi,
\quad
\deru_u(\psi)=u\cdot\psi,
\quad
\sigma_h(\psi)=\psi\cdot h
\]
for $\psi\in {A_q}$.
Then we have
\begin{align}
\label{eq:rel0}
&\ell_1=\deru_1=\sigma_1=\id,
\\
\label{eq:rel1}
&\ell_\varphi\ell_\psi=\ell_{\varphi\psi}
&(\varphi, \psi\in {A_q}),
\\
\label{eq:rel2}
&
\deru_u\deru_{u'}=\deru_{uu'}
&(u, u'\in U_q(\Gg)),
\\
\label{eq:rel3}
&
\sigma_h\sigma_{h'}=\sigma_{hh'}
&(h, h'\in U_q(\Gh)),
\\
\label{eq:rel4}
&\deru_u\ell_\varphi=\sum_{(u)}
\ell_{u_{(0)}\cdot\varphi}\deru_{u_{(1)}}
&(u\in U_q(\Gg),\; \varphi\in {A_q}),
\\
\label{eq:rel5}
&\sigma_h\deru_u=
\deru_u\sigma_h
&(h\in U_q(\Gh),\;  u\in U_q(\Gg)),
\\
\label{eq:rel6}
&\sigma_h\ell_\varphi=
\sum_{(h)}
\chi_\lambda(h_{(0)})\ell_\varphi\sigma_{h_{(1)}}
&(h\in U_q(\Gh),\; \lambda\in\Lambda,\;\varphi\in {A_q}(\lambda)).
\end{align}
It follows that
\begin{align}
\label{eq:4a}
&\ell_\varphi\deru_u=
\sum_{(u)}
\deru_{u_{(1)}}\ell_{(S^{-1}u_{(0)})\cdot\varphi}
&
(u\in U_q(\Gg),\; \varphi\in {A_q}),
\\
\label{eq:6a}
&\ell_\varphi\sigma_h=
\sum_{(h)}\chi_{-\lambda}(h_{(0)})\sigma_{h_{(1)}}
\ell_\varphi
\qquad&(h\in U_q(\Gh),\; \lambda\in\Lambda,\;\varphi\in {A_q}(\lambda)).
\end{align}

By \eqref{eq:z-on-V} we have
\begin{align}
\label{eq:zsigma}
\deru_z=&\sigma_{\Xi(z)}
\qquad&(z\in Z(U_q(\Gg))).
\end{align}

We denote by $\Df$ the subalgebra of $\End_\BF({A_q})$ generated by 
$\ell_\varphi, \deru_u, \sigma_{h}$ 
for 
$\varphi\in {A_q}$, $u\in \Uf(\Gg)$, $h\in U_q(\Gh)$.
\begin{remark}
In \cite{T0} 
we used, instead of $\Df$, a larger subalgebra $D$ of $\End_\BF({A_q})$ generated by 
$\ell_\varphi, \deru_u, \sigma_{h}$ 
for 
$\varphi\in {A_q}$, $u\in U_q(\Gg)$, $h\in U_q(\Gh)$.
The arguments in \cite{T0} for $D$ basically work for our $\Df$ as explained below.
\end{remark}
We sometimes regard ${A_q}$ as a subalgebra of $\Df$ by 
the embedding ${A_q}\ni\varphi\mapsto\ell_\varphi\in \Df$ (see Proposition \ref{prop:A} (i)).
For $\lambda\in\Lambda$ set 
\[
\Df(\lambda)=\{d\in \Df\mid d({A_q}(\mu))\subset {A_q}(\mu+\lambda)\;(\mu\in\Lambda)\}.
\]
Then we have
\[
\Df=\bigoplus_{\lambda\in\Lambda^+}\Df(\lambda),
\]
and $\Df$ turns out to be a $\Lambda$-graded algebra.

Similarly to \cite[Proposition 5.1]{T0} we have the following.
\begin{proposition}
\label{prop:OreD}
For any $w\in W$ the multiplicative set $S_w$ satisfies the left and right Ore conditions in $\Df$.
Moreover, the canonical homomorphism $\Df\to S_w^{-1}\Df$ is injective.
\end{proposition}

\subsection{}
We define an abelian category 
$\Mod(\tDDf_{\CB_q})$ by
\begin{equation}
\label{eq:D}
\Mod(\tDDf_{\CB_q})
=
\Mod_\Lambda(\Df)/
(\Mod_\Lambda(\Df)\cap\Tor_{\Lambda^+}({A_q})),
\end{equation}
where 
$\Mod_\Lambda(\Df)\cap\Tor_{\Lambda^+}({A_q})$
denotes the subcategory of $\Mod_\Lambda(\Df)$
consisting of objects of $\Mod_\Lambda(\Df)$
which belongs to $\Tor_{\Lambda^+}({A_q})$ as an object of $\Mod_\Lambda({A_q})$.
For $\GM, \GN\in\Mod(\tDDf_{\CB_q})$ we denote by
$\Hom_{\tDDf_{\CB_q}}(\GM,\GN)$ the set of morphisms from $\GM$ to $\GN$.

The natural functor 
\begin{equation}
\label{eq:pi D}
\varpi_{\Df}:
\Mod_\Lambda(\Df)\to \Mod(\tDDf_{\CB_q})
\end{equation}
admits a right adjoint 
\begin{equation}
\label{eq:Gamma* D}
\tGamma_{\Df}:\Mod(\tDDf_{\CB_q})
\to \Mod_\Lambda(\Df),
\end{equation}
satisfying $\varpi_{\Df}\circ\tGamma_{\Df}\simto\Id$.
By the universal property  of the localization of categories the forgetful functor 
\[
\For:\Mod_\Lambda(\Df)
\to
\Mod_\Lambda({A_q})
\]
induces an exact functor
\[
\For:\Mod(\tDDf_{\CB_q})
\to
\Mod(\CO_{\CB_q})
\]
such that the following diagram is commutative:
\[
\vcenter{
\xymatrix{
\Mod_\Lambda(\Df)
\ar[r]^{\For}
\ar[d]_{\varpi_{\Df}}
&
\Mod_\Lambda({A_q})
\ar[d]^{\varpi}
\\
\Mod(\tDDf_{\CB_q})
\ar[r]^{\For}
&
\Mod(\CO_{\CB_q}).
}}
\]

The proof of the following result uses Proposition \ref{prop:patch} and Proposition \ref{prop:OreD} among other things 
(see \cite[Lemma 4.5]{T0}).
\begin{proposition}
\label{prop:Gamma-compati}
We have the following commutative diagram of functors:
\[
\vcenter{
\xymatrix{
\Mod(\tDDf_{\CB_q})
\ar[r]^{\For}
\ar[d]_{\tGamma_{\Df}}
&
\Mod(\CO_{\CB_q})
\ar[d]^{\tGamma}
\\
\Mod_\Lambda(\Df)
\ar[r]^{\For}
&
\Mod_\Lambda({A_q}).
}}
\]
\end{proposition}
In the rest of this paper we simply write $\varpi$ and $\tGamma$ for $\varpi_{\Df}$ and $\tGamma_{\Df}$ respectively.

\begin{remark}
\label{rem:Proj3}
Denote by $\DD_\CB$ the ring of differential operators on $\CB$.
We have also a sheaf of rings $\tDD_\CB$ on $\CB$ 
called the sheaf of universally twisted differential operators.
It contains the enveloping algebra $U(\Gh)$ of $\Gh$ in its center, and its specialization 
$\DD_{\CB,\lambda}:=\tDD_\CB\otimes_{U(\Gh)}\BC$ with respect to the algebra homomorphism $U(\Gh) \to \BC$ associated to $\lambda\in\Lambda$ is naturally isomorphic to 
$\CO_\CB(\lambda)\otimes_{\CO_\CB}\DD_\CB\otimes_{\CO_\CB}\CO_\CB(-\lambda)$
(see Remark \ref{rem:Proj1} for the notation).
Our 
$\Mod(\tDDf_{\CB_q})$ 
is a quantum analogue of the category
$\Mod(\tDD_{\CB})$ consisting of quasi-coherent $\tDD_{\CB}$-modules. 
\end{remark}

For $\lambda\in\Lambda$ we denote by 
\[
(\bullet)[\lambda]:
\Mod(\tDDf_{\CB_q})\to 
\Mod(\tDDf_{\CB_q})
\]
the functor induced by
\[ 
(\bullet)[\lambda]:
\Mod_\Lambda(\Df)\to 
\Mod_\Lambda(\Df).
\]

\subsection{}
\begin{lemma}
\label{lem:gpM}
For $\GM\in\Mod(\tDDf_{\CB_q})$ we have a natural isomorphism
\begin{equation}
\label{eq:gpM}
\tGamma(\GM)
\cong
\bigoplus_{\lambda\in\Lambda}
\Hom_{\Df}^\gr(\tGamma(\varpi \Df),\tGamma(\GM)[\lambda])
\end{equation}
of graded $\BF$-modules.
\end{lemma}
\begin{proof}
We have
\begin{align*}
&(\tGamma(\GM))(\lambda)
\cong
(\tGamma(\GM[\lambda]))(0)
\cong
\Hom_{\Df}^\gr(\Df,\tGamma(\GM[\lambda]))
\\
\cong&
\Hom_{\tDDf_{\CB_q}}(\varpi \Df,\GM[\lambda])
\cong
\Hom_{\tDDf_{\CB_q}}(\varpi\tGamma\varpi \Df,\GM[\lambda])
\\
\cong&
\Hom_{\Df}^\gr(\tGamma(\varpi \Df),\tGamma(\GM[\lambda]))
\cong
\Hom_{\Df}^\gr(\tGamma(\varpi \Df),\tGamma(\GM)[\lambda]).
\end{align*}
Here, we used $\varpi\circ \tGamma\simto\Id$ and the fact that 
$\tGamma$ is right adjoint to $\varpi$.
\end{proof}
Especially, we have an isomorphism
\begin{equation}
\label{eq:gpD}
\tGamma(\varpi \Df)
\cong
\bigoplus_{\lambda\in\Lambda}
\Hom_{\Df}^\gr(\tGamma(\varpi \Df),\tGamma(\varpi \Df)[\lambda])
\end{equation}
of $\Lambda$-graded $\BF$-modules.
Note that the right side of \eqref{eq:gpD} is endowed with 
a natural $\Lambda$-graded algebra structure 
given  by the composition of morphisms:
\begin{align*}
&
\Hom_{\Df}^\gr(\tGamma(\varpi {\Df}),\tGamma(\varpi {\Df})[\lambda])
\otimes
\Hom_{\Df}^\gr(\tGamma(\varpi {\Df}),\tGamma(\varpi {\Df})[\mu])
\\
\cong&
\Hom_{\Df}^\gr(\tGamma(\varpi {\Df})[\mu],\tGamma(\varpi {\Df})[\lambda+\mu])
\otimes
\Hom_{\Df}^\gr(\tGamma(\varpi {\Df}),\tGamma(\varpi {\Df})[\mu])
\\
\to&
\Hom_{\Df}^\gr(\tGamma(\varpi {\Df}),\tGamma(\varpi {\Df})[\lambda+\mu]).
\end{align*}
We regard $\tGamma(\varpi {\Df})$ as a 
$\Lambda$-graded algebra  by
\begin{equation}
\label{eq:gpD2}
\tGamma(\varpi {\Df})
=
\left(\bigoplus_{\lambda\in\Lambda}
\Hom_{\Df}^\gr(\tGamma(\varpi {\Df}),\tGamma(\varpi {\Df})[\lambda])
\right)^\op.
\end{equation}
Moreover, for $\GM\in\Mod(\tDDf_{\CB_q})$ 
the right side of \eqref{eq:gpM} is naturally a 
$\Lambda$-graded right module over the right-side of \eqref{eq:gpD} by the composition of morphisms.
Hence we obtain a natural $\Lambda$-graded 
left $\tGamma(\varpi {\Df})$-module structure of 
$\tGamma(\GM)$.

It is easily seen that the canonical homomorphism
\begin{equation}
\label{eq:DGammaD}
{\Df}\to \tGamma(\varpi {\Df})
\end{equation} 
(see \eqref{eq:canon}) is a homomorphism of $\Lambda$-graded algebras.
Moreover, for $\GM\in\Mod(\tDDf_{\CB_q})$
the $\tGamma(\varpi {\Df})$-module structure of 
$\tGamma(\GM)$ described above is compatible with
the ${\Df}$-module structure given by\eqref{eq:Gamma*  D} through \eqref{eq:DGammaD}.
In particular, the functor \eqref{eq:Gamma* D} lifts to 
\begin{equation}
\label{eq:GammaGamma*D}
\tGamma:
\Mod(\tDDf_{\CB_q})\to \Mod_\Lambda(\tGamma(\varpi {\Df}))
\end{equation}
through the canonical algebra homomorphism 
\eqref{eq:DGammaD}.
Note that
we have an equivalence 
\begin{align}
\label{eq:equiv2cat}
\Mod(\tDDf_{\CB_q})
:=&
\Mod_\Lambda({\Df})
/(\Mod_\Lambda({\Df})\cap\Tor_{\Lambda^+}({A_q}))
\\
\nonumber
\cong&
\Mod_\Lambda(\tGamma(\varpi {\Df}))
/(\Mod_\Lambda(\tGamma(\varpi {\Df}))\cap\Tor_{\Lambda^+}({A_q}))
\end{align}
given by the natural functor
$\Mod_\Lambda(\tGamma(\varpi {\Df}))
\to
\Mod_\Lambda({\Df})$
induced by \eqref{eq:DGammaD}, 
and the functor
$\Mod_\Lambda({\Df})\to
\Mod_\Lambda(\tGamma(\varpi {\Df}))$
sending $M\to \tGamma(\varpi M)$.

\begin{proposition}
The canonical algebra homomorphism
\eqref{eq:DGammaD} is injective.
\end{proposition}
\begin{proof}
Recall that ${\Df}$ has been defined as a subalgebra of $\End_\BF({A_q})$.
In particular ${A_q}$ can be regarded as an object of $\Mod_\Lambda({\Df})$.
Hence we have a $\tGamma(\varpi {\Df})$-module structure of $\tGamma(\varpi {A_q})$ compatible with \eqref{eq:DGammaD}.
By Proposition \ref{prop:BW} we have $\tGamma(\varpi {A_q})\cong {A_q}$, and hence we have a $\tGamma(\varpi {\Df})$-module structure of ${A_q}$ compatible with \eqref{eq:DGammaD}.
Then the injectivity of the compsite of 
${\Df}\to\tGamma(\varpi {\Df})\to\End_\BF({A_q})$ implies the desired result.
\end{proof}
\subsection{}
\label{subsec:ad-on-D}
We define the adjoint action of $U_q(\Gg)$ on ${\Df}$ by
\begin{equation}
\label{eq:ad-on-D}
\ad(u)(d)=\sum_{(u)}
\deru_{u_{(0)}}d\deru_{Su_{(1)}}
\qquad
(u\in U_q(\Gg), d\in {\Df}).
\end{equation}
It is integrable and hence we have a right $O(G_q)$-comodule structure of ${\Df}$.
Moreover, ${\Df}$ turns out to be an object of 
$\Mod_\Lambda^\eq({A_q})$ with respect to 
this right $O(G_q)$-comodule structure
and the left ${A_q}$-module structure given by left multiplicaiton.
\subsection{}
For $\lambda\in\Lambda$ we denote by 
$\Mod_{\Lambda}({\Df},\lambda)$ 
the full subcategory of 
$\Mod_\Lambda({\Df})$
consisting of $M\in\Mod_\Lambda({\Df})$
satisfying 
$\sigma_h|_{M(\mu)}=\chi_{\lambda+\mu}(h)\id$ for any $h\in U_q(\Gh)$, $\mu\in\Lambda$.
We set
\begin{equation}
\Mod(\DD^f_{\CB_q,\lambda})
=
\Mod_{\Lambda}({\Df},\lambda)
/
(\Mod_{\Lambda}({\Df},\lambda)\cap
\Tor_{\Lambda^+}({A_q})).
\end{equation}
Then $\Mod(\DD^f_{\CB_q,\lambda})$ is naturally regarded as a full subcategory of $\Mod(\tDDf_{\CB_q})$ (see \cite[Lemma 4.6]{T0}).
For $\GM\in \Mod(\DD^f_{\CB_q,\lambda})$ we have
$\tGamma(\GM)\in \Mod_{\Lambda}({\Df},\lambda)$, and the restriction of $\tGamma$ to $\Mod(\DD^f_{\CB_q,\lambda})$ is right adjoint to
$\varpi:\Mod_{\Lambda}({\Df},\lambda)
\to
\Mod(\DD^f_{\CB_q,\lambda})$ 
(see \cite[Lemma 5.3]{T0}).
\begin{remark}
The category 
$\Mod(\DD^f_{\CB_q,\lambda})$
is a quantum analogue of the category 
$\Mod(\DD_{\CB,\lambda})$ of quasi-coherent $\DD_{\CB,\lambda}$-modules
(see Remark \ref{rem:Proj3}).
\end{remark}
We also define 
$\Mod_\Lambda(\tGamma(\varpi {\Df}),\lambda)$ to be 
the full subcategory of 
$\Mod_\Lambda(\tGamma(\varpi {\Df}))$
consisting of $M\in\Mod_\Lambda(\tGamma(\varpi {\Df}))$
satisfying 
$\sigma_h|_{M(\mu)}=\chi_{\lambda+\mu}(h)\id$ for any $h\in U_q(\Gh)$, $\mu\in\Lambda$ 
with respect to \eqref{eq:DGammaD}.
We see easily that 
\eqref{eq:equiv2cat} induces 
the equivalence
\begin{align}
\label{eq:equiv2catlam}
\Mod(\DD^f_{\CB_q,\lambda})
:=&
\Mod_\Lambda({\Df},\lambda)
/(\Mod_\Lambda({\Df},\lambda)\cap\Tor_{\Lambda^+}({A_q}))
\\
\nonumber
\cong&
\Mod_\Lambda(\tGamma(\varpi {\Df}),\lambda)
/(\Mod_\Lambda(\tGamma(\varpi {\Df}),\lambda)\cap\Tor_{\Lambda^+}({A_q}))
\end{align}
of categories.

\section{$\DD$-modules in the pull-back picture}
\subsection{}
We can easily show the following.
\begin{lemma}
\label{lem:E}
\begin{itemize}
\item[(i)]
We have an $\BF$-algebra structure of
$O(G_q)\otimes U_q(\Gg)$ given by
\[
(\varphi_1\otimes u_1)(\varphi_2\otimes u_2)
=
\sum_{(u_1)}
\varphi_1(u_{1(0)}\cdot \varphi_2)\otimes u_{1(1)}u_2
\]
for $\varphi_1, \varphi_2\in O(G_q)$ 
and
$u_1, u_2\in U_q(\Gg)$.
\item[(ii)]
We have an $\BF$-algebra structure of
$O(G_q)\otimes U_q(\Gg)^\op$  given by
\[
(\varphi_1\otimes u_1^\circ)(\varphi_2\otimes u_2^\circ)
=
\sum_{(u_1)}
\varphi_1(\varphi_2\cdot u_{1(0)})\otimes u_{1(1)}^\circ u_2^\circ
\]
for $\varphi_1, \varphi_2\in O(G_q)$ 
and
$u_1, u_2\in U_q(\Gg)$.
\end{itemize}
\end{lemma}
We denote by $E_L$ (resp.\ $E_R$) the algebra 
$O(G_q)\otimes U_q(\Gg)$ 
(resp.\ $O(G_q)\otimes U_q(\Gg)^\op$) 
given in Lemma \ref{lem:E} (i) 
(resp.\ (ii)).
We identify 
$O(G_q)$ and $U_q(\Gg)$
(resp.\ 
$O(G_q)$ and $U_q(\Gg)^\op$)
with $\BF$-subalgebras of $E_L$ 
(resp.\ $E_R$)
by the embedding
\begin{gather*}
O(G_q)\to E_L
\quad(\varphi\mapsto \varphi\otimes1),
\quad
U_q(\Gg)\to E_L
\quad
(u\mapsto 1\otimes u)
\\
(\text{resp.}
\;\;
O(G_q)\to E_R
\quad
(\varphi\mapsto \varphi\otimes1),
\quad
U_q(\Gg)^\op\to E_R\quad
(u^\circ\mapsto 1\otimes u^\circ)
).
\end{gather*}
In $E_L$ we have
\begin{align}
\label{eq:relEL}
u\varphi=
\sum_{(u)}(u_{(0)}\cdot\varphi)u_{(1)},
\qquad
\varphi u=
\sum_{(u)}u_{(1)}
((S^{-1}(u_{(0)})\cdot\varphi),
\end{align}
and in $E_R$ we have
\begin{align}
\label{eq:relER}
u^\circ\varphi=
\sum_{(u)}
(\varphi\cdot u_{(0)})u_{(1)}^\circ,
\qquad
\varphi^\circ u=
\sum_{(u)}u_{(1)}^\circ
(\varphi\cdot(S(u_{(0)}))
\end{align}
for $u\in U_q(\Gg)$, $\varphi\in O(G_q)$.

Note that $O(G_q)$ is a left $E_L$-module by
\begin{equation}
\label{eq:ELO}
E_L\times O(G_q)
\ni(\varphi u,\psi)\mapsto \varphi(u\cdot\psi)\in O(G_q)
\qquad
(\varphi, \psi\in O(G_q), u\in U_q(\Gg)),
\end{equation}
and a  left $E_R$-module by
\begin{equation}
\label{eq:ERO}
E_R\times O(G_q)
\ni(\varphi u^\circ,\psi)\mapsto \varphi(\psi\cdot u)\in O(G_q)
\qquad
(\varphi, \psi\in O(G_q), u\in U_q(\Gg)).
\end{equation}

We denote by $E_{L,f}$ 
the subalgebra of 
$E_L$ 
generated by 
$\Uf(\Gg)$ 
and $O(G_q)$.
Similarly
we denote by $E_{R,f}$ 
the subalgebra of 
$E_R$
generated by 
$\Uf(\Gg)^\op$ 
and $O(G_q)$.
By \eqref{eq:Uf}
we have
\begin{align*}
&E_{L,f}\cong O(G_q)\otimes \Uf(\Gg),
\qquad
E_{R,f}\cong O(G_q)\otimes \Uf(\Gg)^\op.
\end{align*}

We define linear maps
\begin{equation}
\Psi_{RL}:E_{L,f}\to E_{R,f},
\qquad
\Psi_{LR}:E_{R,f}\to E_{L,f}
\end{equation}
as follows.
Let $\varphi\in O(G_q)$ and 
$u\in \Uf(\Gg)$.
Writing
\[
\ad(w)(u)=\sum_r\langle\varphi_r,w\rangle u_r
\qquad(w\in U_q(\Gg))
\]
for $u_r\in \Uf(\Gg)$, $\varphi_r\in O(G_q)$, we set
\[
\Psi_{RL}(\varphi u)
=\sum_r\varphi\varphi_ru_r^\circ,
\qquad
\Psi_{LR}(\varphi u^\circ)
=\sum_r\varphi(S^{-1}\varphi_r)u_r.
\]
We can easily check the following
\begin{lemma}
\label{lem:fELR}
\begin{itemize}
\item[(i)]
The linear maps $\Psi_{RL}$ and $\Psi_{LR}$ are algebra homomorphisms which are inverse to each other.
In particular, we have an isomorphism
\begin{equation}
\label{eq:fELR}
E_{L,f}\cong E_{R,f}
\end{equation}
of algebras.
\item[(ii)]
$
[\Psi_{RL}(\Uf(\Gg)),\Uf(\Gg)^\op]=
[\Psi_{LR}(\Uf(\Gg)^\op),\Uf(\Gg)]=0$.
\end{itemize}
\end{lemma}

\begin{remark}
Denote by $D_G$ the ring of differential operators on $G$.
The enveloping algebra $U(\Gg)$ 
is naturally identified with the ring of left 
invariant differential operators
by 
$U(\Gg)\ni u\mapsto L_u\in D_G$, 
where 
\[
L_u(\varphi)=u\cdot\varphi
\qquad
(\varphi\in O(G), u\in U(\Gg))
\]
(see \eqref{eq:UgOG} for the notation).
Similarly, 
$U(\Gg)^\op$ 
is naturally identified with the ring of right
invariant differential operators
by 
$U(\Gg)^\op\ni u^\circ\mapsto R_u\in D_G$, 
where 
\[
R_u(\varphi)=\varphi\cdot u
\qquad
(\varphi\in O(G), u\in U(\Gg)).
\]
Moreover, 
identifying $U(\Gg)$ and $U(\Gg)^\op$ with subalgebras of $D_G$ we have isomorphisms
\begin{equation}
\label{eq:LRcong}
O(G)\otimes U(\Gg)\cong
D_G\cong
O(G)\otimes U(\Gg)^\op
\end{equation}
induced by the multiplication of $D_G$.
Note that 
$E_L=O(G_q)\otimes U_q(\Gg)$ 
and 
$E_{L,f}=O(G_q)\otimes \Uf(\Gg)$
(resp.\
$E_R=O(G_q)\otimes U_q(\Gg)^\op$ 
and
$E_{R,f}=O(G_q)\otimes \Uf(\Gg)^\op$)
are  quantum analogues of 
$D_G\cong O(G)\otimes U(\Gg)$
(resp.\
$D_G\cong O(G)\otimes U(\Gg)^\op$).
Moreover, \eqref{eq:fELR} gives  a quantum analogue of \eqref{eq:LRcong}.
\end{remark}
\subsection{}
We define  a left $U_q(\Gg)$-module structure of $E_R$:
\begin{equation}
\label{eq:lac1}
U_q(\Gg)\times E_R\ni (u,d)\mapsto u\lac d\in E_R
\end{equation}
by
\begin{equation}
\label{eq:lac2}
u\lac(\varphi v^\circ)=(u\cdot\varphi)v^\circ
\qquad
(u, v\in U_q(\Gg), \varphi\in O(G_q)).
\end{equation}
We have
\begin{equation}
\label{eq:lac4}
u\lac\Phi_{RL}(v)
=\Phi_{RL}(\ad(u)(v))
\qquad
(u\in U_q(\Gg), v\in E_{L,f}),
\end{equation}
\begin{equation}
\label{eq:lac3}
u\lac(d_1d_2)
=\sum_{(u)}(u_{(0)}\lac d_1)(u_{(1)}\lac d_2)
\qquad
(u\in U_q(\Gg), d_1, d_2\in E_R).
\end{equation}
We also define a right $U_q(\Gg)$-module structure of $E_R$:
\begin{equation}
\label{eq:rac1}
E_R\times U_q(\Gg)\ni (d,u)\mapsto d\rac u\in E_R
\end{equation}
by
\begin{equation}
\label{eq:rac2}
d\rac u=\sum_{(u)}u_{(0)}^\circ d(S^{-1}u_{(1)})^\circ
\qquad
(u\in U_q(\Gg), d\in E_R).
\end{equation}
We have
\begin{equation}
\label{eq:rac3}
v^\circ \rac u=(\ad(S^{-1}u)(v))^\circ
\qquad(u, v\in U_q(\Gg)),
\end{equation}
\begin{equation}
\label{eq:rac4}
\varphi\rac u=
\varphi\cdot u
\qquad(\varphi\in O(G_q),\;u\in U_q(\Gg)),
\end{equation}
\begin{equation}
\label{eq:rac5}
\Psi_{RL}(v)\rac u=\varepsilon(u)\Psi_{RL}(v)
\qquad(v\in \Uf(\Gg), u\in U_q(\Gg)),
\end{equation}
\begin{equation}
\label{eq:rac6}
(d_1d_2)\rac u=
\sum_{(u)}(d_1\rac u_{(0)})(d_2\rac u_{(1)})
\qquad(d_1, d_2\in E_R,\; u\in U_q(\Gg)).
\end{equation}
It is easily seen that $E_R$ is a $U_q(\Gg)$-bimodule with respect to 
\eqref{eq:lac1} and \eqref{eq:rac1}, 
and $E_{R,f}$ is a $U_q(\Gg)$-subbimodule of $E_R$.
\begin{remark}
Define a $G$-biaction on $D_G$:
\[
G\times D_G\times G\ni(g',d,g)\mapsto g'\lac d\rac g
\in D_G
\]
by
\[
(g'\lac d\rac g)(\varphi)
=g'\cdot(d((g')^{-1}\cdot\varphi\cdot g^{-1}))\cdot{g}
\qquad(\varphi\in D_G)
\]
(see Remark \ref{rem:UgOG} for the notation).
By differentiation
we obtain the corresponding $U(\Gg)$-bimodule structure of $D_G$:
\begin{equation}
\label{eq:UgDG}
U(\Gg)\times D_G\times U(\Gg)
\ni(u',d,u)\mapsto u'\lac d\rac u\in D_G.
\end{equation}
Then \eqref{eq:lac1}, \eqref{eq:rac1} are  quantum analogues of \eqref{eq:UgDG}.
\end{remark}

\subsection{}
\label{subsec:DGB}
We define an abelian category 
$\Mod(\DD_{G_q},B_q)$ as follows.
An object is a left $E_R$-module $M$ equipped with a left $O(B_q)$-comodule structure 
$\beta_M:M\to O(B_q)\otimes M$ satisfying the following conditions.
Denote by 
\begin{equation}
\label{eq:triaction}
M\times U_q(\Gb)\to M
\qquad
((m,y)\mapsto m\Uac y)
\end{equation}
the right $U_q(\Gb)$-module structure of $M$ 
corresponding to the left $O(B_q)$-comodule structure of $M$.
Then the conditions for $M$ are:
\begin{itemize}
\item[(A)]
we have 
\[
(dm)\Uac y
=
\sum_{(y)}(d\rac y_{(0)})(m\Uac y_{(1)})
\quad
(d\in E_R, \; m\in M,\; y\in U_q(\Gb)),
\]
i.e.\ 
the linear map $E_R\otimes M\to M$
given by the left $E_R$-module structure of $M$  is a homomorphism of right $U_q(\Gb)$-modules with respect to the right action of $U_q(\Gb)$ on $E_R$ and $M$ given by \eqref{eq:rac1} and \eqref{eq:triaction} respectively,
\item[(B)]
we have
\[
y^\circ m=m\Uac y
\qquad(y\in \tU_q(\Gn), \; m\in M),
\]
i.e.\ the right $\tU_q(\Gn)$-module structure of $M$ induced by the left $E_R$-module structure coincides with the restriction of \eqref{eq:triaction} to $\tU_q(\Gn)$.
\end{itemize}
A morphism in 
$\Mod(\DD_{G_q},B_q)$  is a linear map which respects the left $E_R$-action and the left $O(B_q)$-coaction.
For $M, N\in\Mod(\DD_{G_q},B_q)$ we denote by
$\Hom_{\DD_{G_q},B_q}(M,N)$ 
the set of morphisms from $M$ to $N$.

Note that $O(G_q)$ is an object of 
$\Mod(\DD_{G_q},B_q)$ with respect to the left $E_R$-module structure \eqref{eq:ERO} and the left $O(B_q)$-comodule structure 
$(\res\otimes1)\circ\Delta:O(G_q)\to O(B_q)\otimes O(G_q)$.

For $\mu\in\Lambda$ we define a functor 
\begin{equation}
\label{eq:transPB}
(\bullet)[\mu]:
\Mod(\DD_{G_q},B_q)
\to
\Mod(\DD_{G_q},B_q)
\qquad(M\mapsto M[\mu])
\end{equation}
by
\[
M[\mu]=M\otimes \BF_{-\mu},
\]
where the $E_R$-module structure of $M\otimes \BF_{-\mu}$ is given by the $E_R$-action on the first factor, and the $O(B_q)$-comodule structure is given by the diagonal coaction.
This is well-defined by 
$\Delta(\tU_q(\Gn))\subset U_q(\Gb)\otimes \tU_q(\Gn)$.

\subsection{}
Define a left $E_R$-module $\overline{E}_R$ by
\begin{equation}
\overline{E}_R
=E_R/J,
\qquad 
J=\sum_{y\in\tU_q(\Gn)}
E_R(y^\circ-\varepsilon(y)).
\end{equation}
We have
\[
J=\sum_{y\in\tU_q(\Gn)}
O(G_q)U_q(\Gg)^\op(y^\circ-\varepsilon(y))
\cong
O(G_q)\otimes
\left(\sum_{y\in\tU_q(\Gn)}
U_q(\Gg)^\op(y^\circ-\varepsilon(y))\right),
\]
and hence
\[
\overline{E}_R
\cong
O(G_q)
\otimes
\left(
U_q(\Gg)^\op/
\sum_{y\in\tU_q(\Gn)}
U_q(\Gg)^\op(y^\circ-\varepsilon(y))\right).
\]
We have
\begin{equation}
\label{eq:yd}
\overline{y^\circ d}
=\overline{d\rac y}
\qquad(d\in E_R, y\in\tU_q(\Gn))
\end{equation}
in $\overline{E}_R$.

We see easily the following.
\begin{lemma}
\label{lem:Eb}
\begin{itemize}
\item[(i)]
The left ideal $J$ of $E_R$ is a right $U_q(\Gb)$-submodule of $E_R$ with respect to \eqref{eq:rac1}.
\item[(ii)]
The right $U_q(\Gb)$-module structure of $\overline{E}_R$ given by (i) is integrable so that $\overline{E}_R$ is endowed with a left $O(B_q)$-comodule structure.
\item[(iii)]
The left $E_R$-module $\overline{E}_R$ turns out to be  an object of 
$\Mod(\DD_{G_q},B_q)$ with respect to the left $O(B_q)$-comodule structure given in (ii).
\item[(iv)]
Assume that $J'$ is a left ideal of $E_R$ stable under the right action \eqref{eq:rac1} of $U_q(\Gb)$.
We also assume that the induced right action of $U_q(\Gb)$ on $E_R/J'$ is integrable so that $E_R/J'$ is a left $O(B_q)$-comodule. 
Then $M=E_R/J'$ satisfies the condition (A).
Moreover, if $M=E_R/J'$ satisfies the condition (B), then we have $J'\supset J$.
\end{itemize}
\end{lemma}
\begin{proposition}
\label{prop:sGammaD}
For $\lambda\in\Lambda$ and $M\in\Mod(\DD_{G_q},B_q)$ we have an isomorphism
\[
M^{N_q}(\lambda)\cong
\Hom_{\DD_{G_q},B_q}(\overline{E}_R,M[\lambda])
\]
of $\BF$-modules,
where 
$(\bullet)^{N_q}$ in the left side is taken regarding $M$ as an object of 
$\Mod(\CO_{G_q},B_q)$.
\end{proposition}
\begin{proof}
The proof is reduced to the special case $\lambda=0$ by
$
M^{N_q}(\lambda)
\cong
(M[\lambda])^{N_q}(0)
$.
We have
\begin{align*}
\Hom_{\DD_{G_q},B_q}(\overline{E}_R,M)
=&\Hom_{E_R}(\overline{E}_R,M)
\cap
\Hom_{U_q(\Gb)^\op}(\overline{E}_R,M)
\\
\cong&
\Hom_{E_R}(E_R,M)
\cap
\Hom_{U_q(\Gb)^\op}(E_R,M).
\end{align*}
by Lemma \ref{lem:Eb} (iv).
By $\Hom_{E_R}(E_R,M)\cong M$ we see that 
$\Hom_{\DD_{G_q},B_q}(\overline{E}_R,M)$ is naturally isomorphic to the subspace of $M$ consisting of $m\in M$ satisfying the condition: 
$[d\mapsto dm]\in\Hom_{U_q(\Gb)^\op}({E}_R,M)$.
It is easily seen that this condition holds if and only if  $m\Uac y=\varepsilon(y)m$ for any $y\in U_q(\Gb)$.
We are done.
\end{proof}
Hence we have
\begin{equation}
\label{eq:MN}
M^{N_q}\cong
\bigoplus_{\lambda\in\Lambda}
\Hom_{\DD_{G_q},B_q}(\overline{E}_R,M[\lambda])
\end{equation}
for $M\in\Mod(\DD_{G_q},B_q)$.
Especially, we obtain
\begin{equation}
\label{eq:E-mult}
\overline{E}_R^{N_q}
\cong
\bigoplus_{\lambda\in\Lambda}
\Hom_{\DD_{G_q},B_q}(\overline{E}_R, \overline{E}_R[\lambda]).
\end{equation}
The right side of \eqref{eq:E-mult} is endowed with a 
$\Lambda$-graded algebra structure by 
the multiplication
\begin{align*}
&\Hom_{\DD_{G_q},B_q}(\overline{E}_R, 
\overline{E}_R[\lambda])
\otimes
\Hom_{\DD_{G_q},B_q}(\overline{E}_R, \overline{E}_R[\mu])
\\
\cong&
\Hom_{\DD_{G_q},B_q}(\overline{E}_R[\mu], 
\overline{E}_R[\lambda+\mu])
\otimes
\Hom_{\DD_{G_q},B_q}(\overline{E}_R, \overline{E}_R[\mu])
\\
\to&
\Hom_{\DD_{G_q},B_q}(\overline{E}_R, \overline{E}_R[\lambda+\mu])
\end{align*}
induced by the composition of morphisms.
We regard $\overline{E}_R^{N_q}$ as  a $\Lambda$-graded algebra by
\begin{equation}
\label{eq:E-mult2}
\overline{E}_R^{N_q}
=
\left(
\bigoplus_{\lambda\in\Lambda}
\Hom_{\DD_{G_q},B_q}(\overline{E}_R, \overline{E}_R[\lambda])
\right)^\op.
\end{equation}
Then the right side of \eqref{eq:MN} 
is naturally a right module over the right side of 
\eqref{eq:E-mult} by the composition of morphisms.
We have obtained a left exact functor
\begin{equation}
\label{eq:sGamma*D}
(\bullet)^{N_q}:\Mod(\DD_{G_q},B_q)
\to 
\Mod_\Lambda(
\overline{E}_R^{N_q}).
\end{equation}
By applying \eqref{eq:sGamma*D} to
$O(G_q)\in\Mod(\DD_{G_q},B_q)$ we obtain 
\begin{equation}
\label{eq:AinEN}
{A_q}=O(G_q)^{N_q}\in\Mod_\Lambda(
\overline{E}_R^{N_q}).
\end{equation}
We see easily the following.
\begin{lemma}
\label{lem:sGamma*D}
\begin{itemize}
\item[(i)]
For $d_1, d_2\in E_R$ satisfying 
$\overline{d}_1, \overline{d}_2\in
\overline{E}_R^{N_q}$  we have 
$\overline{d_1d_2}\in\overline{E}_R^{N_q}$, and
$
\overline{d}_1\overline{d}_2=\overline{d_1d_2}$ in
the algebra $\overline{E}_R^{N_q}$.

\item[(ii)]
Let $M\in\Mod(\DD_{G_q},B_q)$ 
and $m\in M^{N_q}$.
For $d\in E_R$ satisfying 
$\overline{d}\in\overline{E}_R^{N_q}$ we have $dm\in M_R^{N_q}$, and
$
\overline{d} m
=dm$ in
the left $\overline{E}_R^{N_q}$-module $M^{N_q}$.
\end{itemize}
\end{lemma}
We sometimes regard ${A_q}$ as a subalgebra of $\EN$ via the  injective algebra homomorphism
\begin{equation}
\label{eq:AEN}
{A_q}\hookrightarrow\EN\qquad
(\varphi\mapsto\overline{\varphi}).
\end{equation}

By \eqref{eq:yd} we have also the following.
\begin{lemma}
\label{lem:EEN}
We have a well-defined linear map
$\overline{E}_R\otimes\EN\to \overline{E}_R$
sending $\overline{d}_1\otimes\overline{d}_2$ to
$\overline{d_1d_2}$.
\end{lemma}
Hence we have a natural $(E_R,\EN)$-bimodule structure of $\overline{E}_R$ given by
\[
d_1\overline{d}\,\overline{d}_2=
\overline{d_1dd_2}
\qquad(\overline{d}\in\overline{E}_R, d_1\in E_R, 
\overline{d}_2\in\EN).
\]
For $K\in\Mod_\Lambda(\EN)$
the left $E_R$-module
$
\overline{E}_R\otimes_{\EN}K
$
is endowed with an integrable  right $U_q(\Gb)$-module structure by
\[
(\overline{d}\otimes k)\Uac y
=
\sum_{(y)}
\tchi_\lambda(y_{(1)})\overline{d\rac y_{(0)}}
\otimes k
\qquad(d\in E_R, k\in K(\lambda), y\in U_q(\Gb))
\]
so that we have $\overline{E}_R\otimes_{\EN}K
\in\Mod(\DD_{G_q},B_q)$.
We obtain a functor 
\begin{equation}
\label{eq:spiE}
\overline{E}_R\otimes_{\EN}(\bullet):
\Mod_\Lambda(\EN)\to\Mod(\DD_{G_q},B_q).
\end{equation}

By Corollary \ref{cor:equivalence} (v)
we have 
a canonical isomorphism 
\begin{equation}
\label{eq:OENE}
O(G_q)\otimes_{A_q}\EN\cong\overline{E}_R
\end{equation} 
in $\Mod(\CO_{G_q},B_q)$, 
and hence we have
\begin{equation}
\label{eq:DO}
\overline{E}_R\otimes_{\EN}K
\cong
O(G_q)\otimes_{A_q}K
\qquad(K\in \Mod_\Lambda(\EN)).
\end{equation}
We see easily that \eqref{eq:DO} is an isomorphism in $\Mod(\CO_{G_q},B_q)$.
Namely, we have the following commutative diagram of functors:
\begin{equation}
\vcenter{
\xymatrix{
\Mod_\Lambda(\EN)
\ar[r]
\ar[d]_{\overline{E}_R\otimes_\EN(\bullet)}
&
\Mod_\Lambda({A_q})
\ar[d]^{O(G_q)\otimes_{A_q}(\bullet)}
\\
\Mod(\DD_{G_q},B_q)
\ar[r]
&
\Mod(\CO_{G_q},B_q),
}}
\end{equation}
where the horizontal arrows are the forgetful functors.

\begin{proposition}
\label{prop:equiv1}
The functor 
\[
\overline{E}_R\otimes_{\EN}(\bullet):
\Mod_\Lambda(\EN)\to\Mod(\DD_{G_q},B_q)
\]
in \eqref{eq:spiE} induces an equivalence
\begin{equation}
\label{eq:equiv1}
\Mod_\Lambda(\overline{E}^{N_q}_R)
/(
\Mod_\Lambda(\overline{E}^{N_q}_R)
\cap
\Tor_{\Lambda^+}({A_q}))
\simto
\Mod(\DD_{G_q},B_q),
\end{equation}
whose inverse is  the composite of 
\[
\Mod(\DD_{G_q},B_q)
\xrightarrow{(\bullet)^{N_q}}
\Mod_\Lambda(\EN)
\xrightarrow{\varpi}
\Mod_\Lambda(\overline{E}^{N_q}_R)
/(
\Mod_\Lambda(\overline{E}^{N_q}_R)
\cap
\Tor_{\Lambda^+}({A_q})).
\]
\end{proposition}
\begin{proof}
By Corrollary \ref{cor:equivalence} (iii), (iv) the functor 
$\overline{E}_R\otimes_\EN(\bullet)$
in \eqref{eq:spiE} is exact and induces 
\[
\Mod_\Lambda(\EN)
/
(\Mod_\Lambda(\overline{E}^{N_q}_R)
\cap \Tor_{\Lambda^+}({A_q}))
\to
\Mod(\DD_{G_q},B_q).
\]
We see easily from Corollary \ref{cor:equivalence} (i) 
that $\varpi\circ((\bullet)^{N_q})$ gives the inverse.
\end{proof}
\subsection{}
\label{subsec:ad-on-E}
The left action \eqref{eq:lac1} of $U_q(\Gg)$ on $E_R$ induces those on $\overline{E}_R$ and 
$\overline{E}_R^{N_q}$.
Note that  $\overline{E}_R^{N_q}$ is  an object of $\Mod_\Lambda^\eq({A_q})$ with respect to the left integrable  $U_q(\Gg)$-module structure
\begin{equation}
\label{eq:ENReq1}
U_q(\Gg)\times \overline{E}_R^{N_q}
\ni(u,\overline{d})
\mapsto
\overline{u\lac d}\in
(\overline{E}_R^{N_q})^\op
\end{equation}
induced by \eqref{eq:lac1},
and the left ${A_q}$-module structure 
\begin{equation}
\label{eq:ENReq2}
{A_q}\times \overline{E}_R^{N_q}
\ni(\psi,\overline{d})
\mapsto
\overline{\psi d}\in
\overline{E}_R^{N_q}
\end{equation}
given by \eqref{eq:AEN}.
\subsection{}
For $\lambda\in\Lambda$ we denote by 
$\Mod(\DD_{G_q},B_q,\lambda)$ 
the full subcategory of 
$\Mod(\DD_{G_q},B_q)$ 
consisting of $M\in \Mod(\DD_{G_q},B_q)$ satisfying 
\[
h^\circ m=\sum_{(h)}\chi_\lambda(h_{(1)})m\Uac h_{(0)}
\qquad
(h\in U_q(\Gh), m\in M).
\]
Namely, 
$\Mod(\DD_{G_q},B_q,\lambda)$ consists of a left $E_R$-module $M$ equipped with a left $O(B_q)$-comodule structure 
$\beta_M:M\to O(B_q)\otimes M$ satisfying the 
condition (A) in \ref{subsec:DGB} and the following condition:
\begin{itemize}
\item[(B;$\lambda$)]
we have
\[
y^\circ m=
\sum_{y}\tchi_\lambda(y_{(1)})m\Uac y_{(0)}
\qquad(y\in U_q(\Gb), \; m\in M),
\]
with respect to \eqref{eq:triaction}.
\end{itemize}

We denote by 
$\Mod_\Lambda(
\overline{E}_R^{N_q},\lambda)$
the full subcategory of 
$\Mod_\Lambda(
\overline{E}_R^{N_q})$ consisting of 
$M\in\Mod_\Lambda(\overline{E}_R^{N_q})$ 
satisfying 
\[
\overline{h^\circ}|_{M(\mu)}
=\chi_{\lambda+\mu}(h)\id
\qquad(h\in U_q(\Gh)).
\]
We see easily that 
\eqref{eq:sGamma*D}, \eqref{eq:spiE} induce
\begin{equation}
\label{eq:sGamma*Dlam}
(\bullet)^{N_q}:
\Mod(\DD_{G_q},B_q,\lambda)
\to
\Mod_\Lambda(\EN,\lambda),
\end{equation}
\begin{equation}
\label{eq:spiElam}
\overline{E}_R\otimes_{\EN}(\bullet):
\Mod_\Lambda(\EN,\lambda)\to\Mod(\DD_{G_q},B_q,\lambda).
\end{equation}

Moreover, 
\eqref{eq:spiElam} induces an equivalence
\begin{equation}
\label{eq:equiv1lam}
\Mod_\Lambda(\overline{E}^{N_q}_R,\lambda)
/(
\Mod_\Lambda(\overline{E}^{N_q}_R,\lambda)
\cap
\Tor_{\Lambda^+}({A_q}))
\simto
\Mod(\DD_{G_q},B_q,\lambda),
\end{equation}
whose inverse is  the composite of 
\begin{align*}
\Mod(\DD_{G_q},B_q,\lambda)
&\xrightarrow{(\bullet)^{N_q}}
\Mod_\Lambda(\EN,\lambda)
\\
&\xrightarrow{\varpi}
\Mod_\Lambda(\overline{E}^{N_q}_R,\lambda)
/(
\Mod_\Lambda(\overline{E}^{N_q}_R,\lambda)
\cap
\Tor_{\Lambda^+}({A_q})).
\end{align*}

\section{Comparison of the two categories of $\DD$-modules}
\subsection{}
The main result of this paper is the following.
\begin{theorem}
\label{thm:main}
\begin{itemize}
\item[(i)]
We have a natural equivalence
\begin{equation}
\label{eq:main-th1}
\Mod(\tDDf_{\CB_q})
\cong
\Mod(\DD_{G_q},B_q)
\end{equation}
of abelian categories.
\item[(ii)]
For any $\lambda\in\Lambda$ the equivalence \eqref{eq:main-th1} induces the equivalence
\begin{equation}
\label{eq:main-th2}
\Mod(\DD^f_{\CB_q,\lambda})
\cong
\Mod(\DD_{G_q},B_q,\lambda)
\end{equation}
of full subcategories.

\end{itemize}
\end{theorem}
We first show (i).
In  view of \eqref{eq:equiv2cat} and \eqref{eq:equiv1}
it  is a consequence of  Proposition \ref{prop:main} below.
\begin{proposition}
We have an isomorphism
\[
\tGamma(\varpi {\Df})\cong\EN
\]
of $\Lambda$-graded $\BF$-algebras.
\label{prop:main}
\end{proposition}
We give a proof of Proposition \ref{prop:main} in the following.
\subsection{}
Set
\[
\zDf={A_q}\otimes \Uf(\Gg)\otimes U_q(\Gh).
\]
It is easily seen that 
$\zDf$ 
is endowed with a $\Lambda$-graded $\BF$-algebra structure by the multiplication
\begin{align*}
&(\varphi\otimes u\otimes h)
(\varphi'\otimes u'\otimes h')
=\sum_{(u), (h)}\chi_\lambda(h_{(0)})
\varphi(u_{(0)}\cdot\varphi')\otimes u_{(1)}u'
\otimes h_{(1)}h'
\\
&\qquad\qquad\qquad
(\varphi\in {A_q}, \varphi'\in {A_q}(\lambda), u,u'\in \Uf(\Gg), h, h'\in U_q(\Gh))
\end{align*}
and the grading $\zDf(\lambda)={A_q}(\lambda)\otimes \Uf(\Gg)\otimes U_q(\Gh)$ ($\lambda\in\Lambda$).
We have a surjective homomorphism
\begin{equation}
\label{eq:ztoD}
p:\zDf\to \Df
\qquad
(\varphi\otimes u\otimes h\mapsto \ell_\varphi\deru_u\sigma_h)
\end{equation}
of $\Lambda$-graded $\BF$-algebras.

Similarly to Proposition \ref{prop:OreD} we see that 
for any $w\in W$ the multiplicative set $S_w$ satisfies the left and right Ore conditions in $\zDf$.
Moreover, the canonical homomorphism $\zDf\to S_w^{-1}(\zDf)$ is injective.

We define an abelian category 
$\Mod(\tDDzf_{\CB_q})$ by
\begin{equation}
\label{eq:D0}
\Mod(\tDDzf_{\CB_q})
=
\Mod_\Lambda(\zDf)/
(\Mod_\Lambda(\zDf)\cap\Tor_{\Lambda^+}({A_q})).
\end{equation}
The natural functor 
\begin{equation}
\label{eq:pi D0}
\varpi_{\zDf}:
\Mod_\Lambda(\zDf)\to \Mod(\tDDzf_{\CB_q})
\end{equation}
admits a right adjoint 
\begin{equation}
\label{eq:Gamma* D0}
\tGamma_{\zDf}:\Mod(\tDDzf_{\CB_q})
\to \Mod_\Lambda(\zDf),
\end{equation}
satisfying $\varpi_{\zDf}\circ\tGamma_{\zDf}\simto\Id$.
By the universal property  of the localization of categories the natural functor 
\[
p^*:\Mod_\Lambda(\Df)
\to
\Mod_\Lambda(\zDf)
\]
given by $p:\zDf\to \Df$ 
induces an exact functor
\[
p^*:\Mod(\tDDf_{\CB_q})
\to
\Mod(\tDDzf_{\CB_q})
\]
such that the following diagram is commutative:
\[
\vcenter{
\xymatrix{
\Mod_\Lambda(\Df)
\ar[r]^{p^*}
\ar[d]_{\varpi}
&
\Mod_\Lambda(\zDf)
\ar[d]^{\varpi_{\zDf}}
\\
\Mod(\tDDf_{\CB_q})
\ar[r]^{p^*}
&
\Mod(\tDDzf_{\CB_q}).
}}
\]

Similarly to Proposition \ref{prop:Gamma-compati}
we have the following.
\begin{proposition}
\label{prop:Gamma-compati2}
We have the following commutative diagram of functors:
\[
\vcenter{
\xymatrix{
\Mod(\tDDf_{\CB_q})
\ar[r]^{p^*}
\ar[d]_{\tGamma}
&
\Mod(\tDDzf_{\CB_q})
\ar[d]^{\tGamma_{\zDf}}
\\
\Mod_\Lambda(\Df)
\ar[r]^{p^*}
&
\Mod_\Lambda(\zDf).
}}
\]
\end{proposition}
In the rest of this paper we simply write $\varpi$ and $\tGamma$ for $\varpi_{\zDf}$ and $\tGamma_{\zDf}$ respectively.

\subsection{}
Now we relate $\zDf$ with $\EN$.
\begin{lemma}
\begin{itemize}
\item[(i)]
For 
$\varphi\in {A_q}(\lambda)$, 
$u\in \Uf(\Gg)$, $h\in U_q(\Gh)$ with 
$\lambda\in\Lambda$ we have
\[
\overline{\Psi_{RL}(\varphi u)h^\circ}
\in \overline{E}_R^{N_q}(\lambda).
\]
\item[(ii)]
The  linear map
\begin{equation}
\label{eq:delta}
\zDf\to\overline{E}_R^{N_q}
\qquad
\left(\varphi\otimes u\otimes h\mapsto
\overline{\Psi_{RL}(\varphi u)h^\circ}
\right)
\end{equation}
is a homomorphism of $\Lambda$-graded $\BF$-algebras.
\end{itemize}
\end{lemma}
\begin{proof}
(i)
By Lemma \ref{lem:sGamma*D} (i) 
it is sufficient to show 
$\overline{\Psi_{RL}(\varphi u)}
\in
\overline{E}_R^{N_q}(\lambda)$ and 
$\overline{h^\circ}\in
\overline{E}_R^{N_q}(0)$.
Hence we have only to show 
\[
\Psi_{RL}(\varphi u)\rac y
=
\tchi_\lambda(y)\Psi_{RL}(\varphi u),
\qquad
\overline{h^\circ\rac y}
=
\varepsilon(y)\overline{h^\circ}
\]
for $y\in U_q(\Gb)$.
The first formula follows from 
$\Psi_{RL}(\varphi u)=\varphi\Psi_{RL}(u)$ and \eqref{eq:rac4}, \eqref{eq:rac5}, \eqref{eq:rac6}.
Set $J'=\Ker(\varepsilon:\tU_q(\Gn)\to\BF)$.
To show the second formula it is sufficient to show  
\[
\ad(z)(h)\in  U_q(\Gb)J'
\qquad(h\in U_q(\Gh), z\in J').
\]
We may assume $h=k_\lambda$ for $\lambda\in\Lambda$ and $z\in \tU_q(\Gn)_{-\gamma}$ for 
$\gamma\in Q^+\setminus\{0\}$.
Then we have
\begin{align*}
\ad(z)(h)=&
\sum_{(z)}z_{(0)}k_\lambda(Sz_{(1)})
\in
\sum_{(z)}\varepsilon(z_{(1)})z_{(0)}k_\lambda+U_q(\Gb)J'
=zk_\lambda+U_q(\Gb)J'
\\
=&
q^{(\lambda,\gamma)}
k_\lambda z+J'=J'.
\end{align*}

(ii) By (i) the linear map \eqref{eq:delta} respects the $\Lambda$-grading.
Hence it is sufficient to  show that  \eqref{eq:delta}  is a homomorphism of $\BF$-algebras.
The restriction of \eqref{eq:delta} to the 
subalgebra ${A_q}\otimes \Uf(\Gg)\otimes 1$ of $\zDf$ is a homomorphism of $\BF$-algebras by 
Lemma \ref{lem:sGamma*D} (i).
Hence it is sufficient to show 
\[
h^\circ\Psi_{RL}(\varphi u)=\sum_{(h)}
\chi_\lambda(h_{(0)})\Psi_{RL}(\varphi u)h_{(1)}^\circ
\qquad
(h\in U_q(\Gh),
\varphi\in {A_q}(\lambda), 
u\in \Uf(\Gg))
\]
in $E_R$.
This easily follows from
\eqref{eq:relER} and Lemma \ref{lem:fELR} (ii).
\end{proof}
We define a $\Lambda$-graded $\BF$-algebra 
$\iDf$ by
\begin{equation}
\label{eq:iDf}
\iDf=\Image(\zDf\to\overline{E}_R^{N_q}).
\end{equation}
We obtain a surjective homomorphism
\begin{equation}
\label{eq:ztoi}
p':\zDf\to \iDf
\end{equation}
and an injective homomorphism
\begin{equation}
\label{eq:itoE}
\jmath:\iDf
\to
\overline{E}_R^{N_q}
\end{equation}
of $\Lambda$-graded $\BF$-algebras.

Note that $\zDf$ is regarded as an object of $\Mod_\Lambda^\eq({A_q})$ by the ${A_q}$-module structure 
given by
\[
\psi(\varphi\otimes u\otimes h)
=
\psi\varphi\otimes u\otimes h
\quad(\psi, \varphi\in {A_q}, u\in \Uf(\Gg), h\in U_q(\Gh))
\]
and the left integrable 
$U_q(\Gg)$-module structure  given by
\[
x(\varphi\otimes u\otimes h)
=
\sum_{(x)}
x_{(0)}\cdot\varphi\otimes
\ad(x_{(1)})(u)\otimes h
\quad(x\in U_q(\Gg), \varphi\in {A_q}, u\in \Uf(\Gg), h\in U_q(\Gh)),
\]
for which 
$p:\zDf\to\Df$ in
\eqref{eq:ztoD} is a morphism in 
$\Mod_\Lambda^\eq({A_q})$
(see \ref{subsec:ad-on-D}).
Note also that
$\zDf\to\EN$ in 
\eqref{eq:delta} is a morphism in $\Mod_\Lambda^\eq({A_q})$ 
(see \ref{subsec:ad-on-E})
by \eqref{eq:lac2}, \eqref{eq:lac4}, \eqref{eq:lac3},
and hence $\iDf$ is a subobject of $\EN$ in $\Mod_\Lambda^\eq({A_q})$.

\begin{proposition}
\label{prop:Disom}
The surjective homomorphisms 
$p:\zDf\to\Df$ in
\eqref{eq:ztoD} and 
$p':\zDf\to\iDf$ in \eqref{eq:ztoi}  induce a surjective homomorphism
\begin{equation}
\label{eq:itoD}
\pi:\iDf\to \Df
\end{equation}
of $\Lambda$-graded $\BF$-algebras.
\end{proposition}
\begin{proof}
By \eqref{eq:AinEN} 
we have an action of $\iDf$ on ${A_q}$.
By Lemma \ref{lem:sGamma*D} (ii) the actions of $\zDf$ on ${A_q}$ given by $p'$ and $p$ coincide.
Since $\Df$ was defined as a subalgebra of $\End_\BF({A_q})$, we obtain the desired result.
\end{proof}

Note that $\pi:\iDf\to \Df$ is a morphism in  both  $\Mod_\Lambda(\zDf)$ and $\Mod_\Lambda^\eq({A_q})$.
Applying $\tGamma\circ\varpi$ to 
$\pi$ we obtain a morphism
\begin{equation}
\label{eq:main1}
\pi_*:\tGamma(\varpi \iDf)
\to
\tGamma(\varpi D)
\end{equation}
in both  $\Mod_\Lambda(\zDf)$ and 
$\Mod_\Lambda^\eq({A_q})$.

By Corollary \ref{cor:equivalence} (i), (ii) 
we have
\[
\tGamma(\varpi\overline{E}^{N_q}_R)
\cong
\overline{E}^{N_q}_R.
\]
Therefore, applying $\tGamma\circ \varpi$ 
to $\jmath:\iDf\to\EN$ we obtain a morphism
\begin{equation}
\label{eq:main2}
\jmath_*:\tGamma(\varpi \iDf)
\to
\EN
\end{equation}
in both $\Mod_\Lambda(\zDf)$ and 
$\Mod_\Lambda^\eq({A_q})$.

We will prove Proposition \ref{prop:main} by showing that $\pi_*$ and $\jmath_*$ are
 isomorphisms of $\Lambda$-graded $\BF$-algebras.
 \subsection{}
 We first establish the following weaker fact.
\begin{proposition}
\label{prop:Disom2}
The homomorphisms $\pi:\iDf\to \Df$ and $\jmath:\iDf\to \EN$ induce isomorphisms
\[
\varpi(\iDf)\cong\varpi(\Df),
\qquad
\varpi(\iDf)
\cong
\varpi(\overline{E}_R^{N_q})
\]
in $\Mod(\CO_{\CB_q})$.
\end{proposition}
\begin{proof}
Let us first show that $\pi$ induces 
$\varpi(\iDf)\cong\varpi({\Df})$.
Note that $\iDf\to {\Df}$ is a morphism in
$\Mod^\eq_\Lambda({A_q})$.
Hence by Proposition \ref{prop:iota} 
 it is sufficient to verify that 
$\BF\otimes_{A_q}\iDf\to \BF\otimes_{A_q}{\Df}$ is bijective.
Consider the commutative diagram 
\[
\vcenter{
\xymatrix{
\iDf
\ar[r]^{\jmath}
\ar@{->>}[d]_{\pi}
&
\overline{E}_R^{N_q}
\ar[d]^{a}
\\
{\Df}
\ar@{^{(}->}[r]
&
\End_\BF({A_q})
}}
\]
of left ${A_q}$-modules (see \eqref{eq:AinEN}).
Since we have a canonical homomorphism 
$\BF\otimes_{A_q}\End_\BF({A_q})\to\Hom_\BF({A_q},\BF)$, we obtain a commutative diagram
\[
\vcenter{
\xymatrix{
\BF\otimes_{A_q}\iDf
\ar[r]^{\overline{\jmath}}
\ar[d]_{\overline{\pi}}
&
\BF\otimes_{A_q}\overline{E}_R^{N_q}
\ar[d]^{\overline{a}}
\\
\BF\otimes_{A_q}{\Df}
\ar[r]
&
\Hom_\BF({A_q},\BF).
}}
\]
By Proposition \ref{prop:iota} 
$\overline{\pi}$ is surjective and $\overline{\jmath}$ is injective.
To show that $\overline{\pi}$ is injective, it is sufficient to show that $\overline{a}$ is injective.
By 
\eqref{eq:OENE}
we have
\[
\BF\otimes_{A_q}
\overline{E}_R^{N_q}
\cong
\BF\otimes_{O(G_q)}\overline{E}_R
\cong
U_q(\Gg)^\op/
\sum_{y\in\tU_q(\Gn)}
U_q(\Gg)^\op(y^\circ-\varepsilon(y))
\cong U_q(\Gb^+)^\op.
\]
Under this identification $\overline{a}$ is given by
\begin{equation}
\label{eq:b{A_q}F}
U_q(\Gb^+)^\op
\ni
x^\circ\mapsto[\varphi\mapsto\langle\varphi,x\rangle]
\in\Hom_\BF({A_q},\BF).
\end{equation}
Assume that $x^\circ$ for $x\in U_q(\Gb^+)$ belongs to the kernel of \eqref{eq:b{A_q}F}.
By \eqref{eq:Alam2} we have $v^*_\lambda x=0$ for any $\lambda\in\Lambda^+$, where $v^*_\lambda$ is a non-zero element of $V^*(\lambda)_\lambda$.
Write $x=\sum_{r}h_rx_r$ for $h_r\in U_q(\Gh)$, $x_r\in U_q(\Gn^+)$.
We may assume that $\{x_r\}$ is linearly independent.
Then we have 
$\sum_r\chi_\lambda(h_r)v^*_\lambda x_r=0$ for any $\lambda\in\Lambda^+$.
Denote by $\Lambda^+_{>N}$ the set of $\lambda\in\Lambda^+$ satisfying 
$\langle\lambda,\alpha_i^\vee\rangle>N$ for $\forall i\in I$.
It is well-known that there exists some $N>0$ such that 
for any $\lambda\in \Lambda^+_{>N}$  the subset $\{v^*_\lambda x_r\}$ of $V^*(\lambda)$ is  linearly independent
(see for example the proof of \cite[Proposition 5.11]{Jan}).
Hence we have $\chi_\lambda(h_r)=0$ for any $r$ and $\lambda\in\Lambda_{>N}^+$.
From this we easily obtain $h_r=0$ for any $r$.
The injectivity of \eqref{eq:b{A_q}F} is verified.
We have shown $\varpi({\Df})\cong\varpi(\iDf)$.

Let us show that $\jmath$ induces
$\varpi(\iDf)\cong
\varpi(\overline{E}_R^{N_q})$.
Since $\varpi$ is an exact functor, we have 
the injectivity of $\varpi(\iDf)\to
\varpi(\overline{E}_R^{N_q})$.
To show the surjectivity it is sufficient to 
prove that 
$\varpi(\zDf)\to
\varpi(\overline{E}_R^{N_q})$ 
is surjective.
This is equivalent to 
the surjectivity of 
$O(G_q)\otimes_{A_q}\zDf\to
O(G_q)\otimes_{A_q}\overline{E}_R^{N_q}$
by \eqref{eq:AZBK1}.
By \eqref{eq:OENE}
it is sufficient to show the surjectivity of the linear map 
\[
O(G_q)\otimes \Uf(\Gg)\otimes U_q(\Gh)\to E_R
\qquad
(\varphi\otimes u\otimes h\mapsto
\Phi_{RL}(\varphi u)h^\circ).
\]
This follows from Lemma \ref{lem:gen}.
\end{proof}

\subsection{}
Now we complete the proof of Theorem \ref{thm:main} (i).
By Proposition \ref{prop:Disom2} 
the linear maps
\[
\pi_*: \tGamma(\varpi \iDf)\to\tGamma(\varpi {\Df}), 
\qquad
\jmath_*: \tGamma(\varpi \iDf)\to\EN
\]
are isomorphism of $\Lambda$-graded $\zDf$-modules.
Hence we have an isomorphism
\[
F=\jmath_*\circ\pi_*^{-1}:
\tGamma(\varpi {\Df})\simto\EN
\]
of $\Lambda$-graded $\zDf$-module.
We need to show that $F$ is an isomorphism of 
$\Lambda$-graded $\BF$-algebras.

Recall that the multiplication of $\tGamma(\varpi {\Df})$ is defined by the identification
\begin{equation}
\label{eq:mult1}
\tGamma(\varpi {\Df})
\cong
\left(\bigoplus_{\lambda\in\Lambda}
 \Hom^\gr_{{\Df}}
 (\tGamma(\varpi {\Df}),\tGamma(\varpi {\Df})[\lambda])
 \right)^\op.
 \end{equation}
We have also an
isomorphism
\begin{equation}
\label{eq:mult2}
\EN\cong
\left(
\bigoplus_{\lambda\in\Lambda}
\Hom_{\EN}^\gr(\EN,\EN[\lambda])
\right)^\op
\end{equation}
of $\Lambda$-graded $\BF$-algebras.
Recall also 
 \[
 \Hom^\gr_{{\Df}}
 (\tGamma(\varpi {\Df}),\tGamma(\varpi {\Df})[\lambda])
 \cong
  \Hom^\gr_{{\Df}}
 ({\Df},\tGamma(\varpi {\Df})[\lambda])
 \]
for $\lambda\in\Lambda$ (see the proof of Lemma \ref{lem:gpM}).
Hence
for $z\in\tGamma(\varpi {\Df})(\lambda)$ 
the corresponding element
$f_z\in 
\Hom_{{\Df}}^\gr(\tGamma(\varpi {\Df}),\tGamma(\varpi {\Df})[\lambda])$ via 
 \eqref{eq:mult1}
 is uniquely characterized by $f_z(1)=z$.
Similarly, for $z'\in\EN(\lambda)$ 
the corresponding element
$g_{z'}\in 
\Hom_{\EN}^\gr(\EN,\EN[\lambda])$
via \eqref{eq:mult2} is uniquely characterized by $g_{z'}(1)=z'$.
The linear isomorphism
\begin{equation}
\label{eq:mult3}
\left(\bigoplus_{\lambda\in\Lambda}
 \Hom^\gr_{{\Df}}
 (\tGamma(\varpi {\Df}),\tGamma(\varpi {\Df})[\lambda])
 \right)^\op
\simto\left(
\bigoplus_{\lambda\in\Lambda}
\Hom_{\EN}^\gr(\EN,\EN[\lambda])
\right)^\op
\end{equation}
corresponding to $F$ sends 
$f_z\in \Hom^\gr_{{\Df}}(\tGamma(\varpi {\Df}),\tGamma(\varpi {\Df})[\lambda])$ 
for $z\in \tGamma(\varpi {\Df})(\lambda)$ to
$g_{F(z)}\in \Hom_{\EN}^\gr(\EN,\EN[\lambda])$.
We need to show that \eqref{eq:mult3} respects the multiplication defined by the composition of morphisms.
For that it is sufficient to show that the following diagram is commutative:
\[
\vcenter{
\xymatrix{
\tGamma(\varpi {\Df})
\ar[r]^{f_z}
\ar[d]_{F}^\cong
&
\tGamma(\varpi {\Df})[\lambda]
\ar[d]^{F}_\cong
\\
\EN
\ar[r]_{g_{F(z)}}
&
\EN[\lambda].
}
}
\]
By $F(1)=1$ we have $(F\circ f_z)(1)=F(z)=(g_{F(z)}\circ F)(1)$.
Hence  we obtain
$F\circ f_z=g_{F(z)}\circ F$
by
\begin{align*}
 &\Hom^\gr_{\zDf}
 (\tGamma(\varpi {\Df}),\EN[\lambda])
 \cong
 \Hom^\gr_{\zDf}
 (\tGamma(\varpi {\Df}),\tGamma(\varpi\EN)[\lambda])
 \\
  \cong&
 \Hom_{\tDDzf_{\CB_q}}
 (\varpi\tGamma(\varpi {\Df}),\varpi\EN[\lambda])
 \cong
  \Hom_{\tDDzf_{\CB_q}}
 (\varpi {\Df},\varpi\EN[\lambda])
 \\
\cong&
\Hom^\gr_{\zDf}
({\Df},\tGamma(\varpi\EN)[\lambda])
\cong
\Hom^\gr_{\zDf}
({\Df},\EN[\lambda])
\\
\hookrightarrow&\EN(\lambda).
 \end{align*}

We have shown Theorem \ref{thm:main}(i).
Then Theorem \ref{thm:main} (ii) is clear in view of \eqref{eq:equiv2catlam} and \eqref{eq:equiv1lam}.
The proof of 
Theorem \ref{thm:main} is now complete.
\begin{remark}
By the same arguments we can show the analogue of Theorem \ref{thm:main} in the setting
where $q$ is specialized to a root of unity (see \cite{T1}, \cite{T2}, \cite{T3}, \cite{T4}).
Details are omitted.
\end{remark}

\section{The Beilinson-Bernstein correspondence}
\subsection{}
\label{subsec:lam}
For $\lambda\in\Lambda$ we define 
$\Df_\lambda
\in\Mod_{\Lambda}(\Df,\lambda)$ by 
\begin{equation}
\Df_\lambda
=\Df/\sum_{h\in U_q(\Gh)}\Df(\sigma_h-\chi_\lambda(h)).
\end{equation}
Since $\sigma_h$ for $h\in U_q(\Gh)$ belongs to the center of $\Df(0)$,
we have a natural $\BF$-algebra structure of 
\[
\Df_\lambda(0)=
\Df(0)/\sum_{h\in U_q(\Gh)}\Df(0)(\sigma_h-\chi_\lambda(h)).
\]

It is easily seen that 
the functor
\begin{equation}
\label{eq:Gammalam1}
\Gamma:\Mod(\DD^f_{\CB_q,\lambda})
\to \Mod(\Df_\lambda(0))
\qquad
(\GM\mapsto \tGamma(\GM)(0))
\end{equation}
is right adjoint to 
\begin{equation}
\label{eq:Loclam}
\Loc:
\Mod(\Df_\lambda(0))\to
\Mod(\DD^f_{\CB_q,\lambda})
\qquad
(M\mapsto \varpi(\Df\otimes_{\Df(0)}M)).
\end{equation}

Similarly to Lemma \ref{lem:gpM} we have an isomorphism
\begin{equation}
\label{eq:GammaDlam1}
\Gamma(\GM)
\cong\Hom_{\Df}^\gr(\tGamma(\varpi \Df_\lambda),
\tGamma(\GM))
\qquad(\GM\in\Mod_\Lambda(\Df,\lambda))
\end{equation}
of $\BF$-modules.
In particular, we have
\begin{equation}
\label{eq:GammaDlam2}
\Gamma(\varpi \Df_\lambda)
\cong\End_{\Df}^\gr(\tGamma(\varpi \Df_\lambda)).
\end{equation}
The right side of \eqref{eq:GammaDlam2} is naturally an $\BF$-algebra by the composition of morphisms.
We regard $\Gamma(\varpi \Df_\lambda)$ as an $\BF$-algebra by the identification
\begin{equation}
\label{eq:GammaDlam3}
\Gamma(\varpi \Df_\lambda)
=\End_{\Df}^\gr(\tGamma(\varpi \Df_\lambda))^\op.
\end{equation}
The the natural homomorphism
\begin{equation}
\label{eq:DlamGamma}
\Df_\lambda(0)\to \Gamma(\varpi \Df_\lambda)
\end{equation}
of $\Df(0)$-modules turns out to be 
a homomorphism of $\BF$-algebras.

Since the right side of \eqref{eq:GammaDlam1} is naturally a right module over the right side of  \eqref{eq:GammaDlam2}
we see that 
$\Gamma(\GM)$ for 
$\GM\in\Mod(\DD^f_{\CB_q,\lambda})$ is a left 
$\Gamma(\varpi \Df_\lambda)$-module.
This gives a functor
\begin{equation}
\label{eq:Gammalam2}
\Gamma:\Mod(\DD^f_{\CB_q,\lambda})
\to \Mod(\Gamma(\varpi \Df_\lambda)),
\end{equation}
which is a lift of \eqref{eq:Gammalam1}.

\subsection{}
For $\lambda\in\Lambda$ we define an ideal $I_\lambda$ of $\Uf(\Gg)$ by
\begin{equation}
I_\lambda=
\sum_{z\in Z(U_q(\Gg))}
\Uf(\Gg)(z-\chi_\lambda(\Xi(z))).
\end{equation}
The following result is established in the course of the proof of \cite[Theorem 5.2]{T0}. 
\begin{theorem}
\label{thm:isom}
For $\lambda\in\Lambda$
we have isomorphisms
\[
\Uf(\Gg)/I_\lambda
\cong
\Df_\lambda(0)
\cong
\Gamma(\varpi \Df_\lambda)
\]
of algebras.
Here, the first isomorphism is induced by 
$\Uf(\Gg)\to \Df(0)$ ($u\mapsto\deru_u$), and 
the second isomorphism is \eqref{eq:DlamGamma}.
\end{theorem}

We briefly give a sketch of the proof of Theorem \ref{thm:isom}. 
As in \cite[5.7]{T0} we have a sequence 
\[
\Uf(\Gg)/I_\lambda
\xrightarrow{\alpha}
\Df_\lambda(0)
\xrightarrow{\beta}
\Gamma(\varpi \Df_\lambda)
\xrightarrow{\gamma}
\End(T_r(\lambda))^\op
\]
of algebra homomorphisms such that $\alpha$ is surjective and $\gamma$ is injective.
Here, $T_r(\lambda)$ is the right Verma module with 
highest weight $\lambda$ given by 
\[
T_r(\lambda)=
U_q(\Gg)/
\sum_{y\in U_q(\Gb)}(y-\tchi_\lambda(y))U_q(\Gg).
\]
By Joseph's result \cite[8.3.9]{JoB} on Verma module annihilators we see that $\gamma\circ\beta\circ\alpha$ is injective.
Hence $\alpha$ is an isomorphism and $\beta$ is injective.
Define the adjoint action of $U_q(\Gg)$  on 
$\End(T_r(\lambda))^\op$  by
\[
\ad(u)(a)=\sum_{(u)}F(u_{(0)})aF(Su_{(1)})
\qquad (u\in U_q(\Gg), a\in \End(T_r(\lambda))^\op),
\]
where 
$F:U_q(\Gg)\to \End(T_r(\lambda))^\op$ 
is the canonical algebra homomorphism.
By Joseph's result \cite[Theorem 8.3.9 (ii)]{JoB}, 
which is a quantum analogue of a Theorem of Conze-Berline and Duflo, 
the image of $\gamma\circ\beta\circ\alpha$ coincides with 
\[
(\End(T_r(\lambda))^\op)^f
=
\{
a\in
\End(T_r(\lambda))^\op\mid 
\dim\ad(U_q(\Gg))(a)<\infty\}.
\]
We see easily that $\Df_\lambda(0)\in\Mod_\Lambda^\eq({A_q})$ 
with respect to the adjoint action of $U_q(\Gg)$ on $\Df_\lambda$
induced by \eqref{eq:ad-on-D}, and hence 
the image of $\gamma\circ\beta$ is contained in $(\End(T_r(\lambda))^\op)^f$.
Therefore $\beta$ is also an isomorphism. 
\begin{remark}
\begin{itemize}
\item[(i)]
The proof of Theorem \ref{thm:isom}, which is sketched above, is analogous to that of the corresponding fact for Lie algebras 
given by Borho and Brylinsky \cite{BoB}.
\item[(ii)]
Theorem \ref{thm:isom} also follows from \cite[Theorem 8.1]{T3}.
It is also shown there that $R^i\Gamma(\varpi \Df_\lambda)=0$ for $i>0$.
The proof of \cite[Theorem 8.1]{T3} is rather complicated in order to deal with the situation where $q$ is specialized to a root of unity whose order is a power of a prime.
The arguments in \cite{T3} becomes much simpler
if we restrict ourselves to the case where $q$ is transcendental.
Details are left to the readers.
\end{itemize}
\end{remark}

As in \cite{T0} we have the following Beilinson-Bernstein type equivalence of categories.
\begin{theorem}
\label{thm:BB}
For $\lambda\in\Lambda^+$ the functor 
\[
\Gamma:
\Mod(\DD^f_{\CB_q,\lambda})\to \Mod(\Uf(\Gg)/I_\lambda)
\]
in \eqref{eq:Gammalam1} gives an equivalence of categories, whose inverse is
\[
\Loc:\Mod(\Uf(\Gg)/I_\lambda)
\to
\Mod(\DD^f_{\CB_q,\lambda})
\]
in \eqref{eq:Loclam}.
\end{theorem}
The original proof in \cite{BB} of the Beilinson-Bernstein correspondence for Lie algebras works for our case once Theorem \ref{thm:isom} is established.
We omit the details here (see \cite{LR}, \cite{T0}).

\subsection{}
We now describe the corresponding results for
$\Mod(\DD_{G_q}, B_q,\lambda)$.

Define a left $E_R$-module $\overline{E}_{R,\lambda}$ by
\begin{equation}
\overline{E}_{R,\lambda}=
E_R/J_\lambda,
\qquad
J_\lambda=\sum_{y\in U_q(\Gb)}E_R(y^\circ-\tchi_\lambda(y))
\end{equation}
The right $U_q(\Gg)$-module structure \eqref{eq:rac1} of $E_R$ induces an integrable right $U_q(\Gb)$-module structure of $\overline{E}_{R,\lambda}$, 
by which $\overline{E}_{R,\lambda}$ turns out to be an object of $\Mod(\DD_{G_q},B_q,\lambda)$.
Similarly to Proposition \ref{prop:sGammaD},
 we have an isomorphism
\begin{equation}
\label{eq:MBE}
M^{B_q}\cong
\Hom_{\DD_{G_q},B_q}(\overline{E}_{R,\lambda},M)
\qquad(M\in\Mod(\DD_{G_q},B_q,\lambda))
\end{equation}
of $\BF$-modules,
where 
$(\bullet)^{B_q}$ in the left side is taken regarding $M$ as an object of 
$\Mod(\CO_{G_q},B_q)$.
Especially, we obtain
\begin{equation}
\label{eq:E-multB}
\overline{E}_{R,\lambda}^{B_q}
\cong
\End_{\DD_{G_q},B_q}(\overline{E}_{R,\lambda}).
\end{equation}
The right side of \eqref{eq:E-multB} is endowed with an algebra structure by 
 the composition of morphisms.
We regard $\overline{E}_{R,\lambda}^{B_q}$ as  an algebra by
\begin{equation}
\label{eq:E-multB2}
\overline{E}_{R,\lambda}^{B_q}
=
\End_{\DD_{G_q},B_q}(\overline{E}_{R,\lambda})^\op.
\end{equation}
Then the right side of \eqref{eq:MBE} 
is naturally a right module over the right side of 
\eqref{eq:E-multB} by the composition of morphisms.
We have obtained a left exact functor
\begin{equation}
\label{eq:sGamma*Dlam2}
(\bullet)^{B_q}:\Mod(\DD_{G_q},B_q,\lambda)
\to 
\Mod(
\overline{E}_{R,\lambda}^{B_q}).
\end{equation}
\begin{proposition}
\begin{itemize}
\item[(i)]
The object $\varpi \Df_\lambda\in \Mod(\DD^f_{\CB_q,\lambda})$ corresponds to
$\overline{E}_{R,\lambda}\in \Mod(\DD_{G_q}, B_q,\lambda)$
under the category equivalence
\[
\Mod(\DD^f_{\CB_q,\lambda})\cong
\Mod(\DD_{G_q}, B_q,\lambda)
\]
of Theorem \ref{thm:main}.
\item[(ii)]
We have a canonical isomorphism
\[
\Gamma(\varpi \Df_\lambda)
\cong
\overline{E}_{R,\lambda}^{B_q}
\]
of $\BF$-algebras.
\item[(iii)]
We have the following commutative diagram of functors:
\[
\vcenter{
\xymatrix{
\Mod(\DD_{G_q}, B_q,\lambda)
\ar[r]^-\cong
\ar[d]_{(\bullet)^{B_q}}
&
\Mod(\DD^f_{\CB_q,\lambda})
\ar[d]^{\Gamma}
   \\
\Mod(
\overline{E}_{R,\lambda}^{B_q})
\ar[r]
&
\Mod(\Df_\lambda(0)),
}}
\]
where the lower horizontal arrow is given by the algebra homomorphism $\Df_\lambda(0)\to \overline{E}_{R,\lambda}^{B_q}$.
\end{itemize}
\end{proposition}
\begin{proof}
Recall that the  category equivalence in (i) is
given by the composite of 
\begin{align*}
&\Mod(\DD^f_{\CB_q,\lambda})
\\
=&
\Mod_\Lambda(\Df,\lambda)
/(\Mod_\Lambda(\Df,\lambda)\cap\Tor_{\Lambda^+}({A_q}))
\\
\cong&
\Mod_\Lambda(\tGamma(\varpi \Df),\lambda)
/(\Mod_\Lambda(\tGamma(\varpi \Df),\lambda)\cap\Tor_{\Lambda^+}({A_q}))
\\
\cong&
\Mod_\Lambda(\EN,\lambda)
/(\Mod_\Lambda(\EN,\lambda)\cap\Tor_{\Lambda^+}({A_q}))
\\
\cong&
\Mod(\DD_{G_q}, B_q,\lambda).
\end{align*}
It is easily seen that the object of 
$\Mod_\Lambda(\EN,\lambda)
/(\Mod_\Lambda(\EN,\lambda)\cap\Tor_{\Lambda^+}({A_q}))$
corresponding to $\varpi \Df_\lambda\in
\Mod(\DD^f_{\CB_q,\lambda})$ is given by 
\[
\EN/
\sum_{h\in U_q(\Gh)}\EN(\overline{h^\circ}-\chi_\lambda(h))
\in
\Mod_\Lambda(\EN,\lambda).
\]
Hence (i) follows from
\[
\overline{E}_R
\otimes_{\EN}
\left(
\EN/
\sum_{h\in U_q(\Gh)}\EN(\overline{h^\circ}-\chi_\lambda(h))
\right)
\cong
\overline{E}_R/
\sum_{h\in U_q(\Gh)}\overline{E}_R
(\overline{h^\circ}-\chi_\lambda(h))
\cong\overline{E}_{R,\lambda}.
\]
Then (ii) and (iii) are clear from (i).
\end{proof}

By Theorem \ref{thm:isom} and 
Theorem \ref{thm:BB} we recover the following results  in \cite{BK}.
\begin{theorem}
\label{thm:isom2}
For $\lambda\in\Lambda$
we have an isomorphism
\[
\Uf(\Gg)/I_\lambda
\cong
\overline{E}^{B_q}_{R,\lambda}
\]
of algebras.
\end{theorem}
\begin{theorem}
\label{thm:BB2}
For $\lambda\in\Lambda^+$ the functor
\[
(\bullet)^{B_q}:
\Mod(\DD_{G_q},B_q,\lambda)\to \Mod(\Uf(\Gg)/I_\lambda)
\]
gives an equivalence of categories.
\end{theorem}

\bibliographystyle{unsrt}

\end{document}